\documentclass{amsart}

\usepackage{palatino, mathpazo}
\usepackage{amsfonts}
\usepackage{amsmath}
\usepackage{amssymb,latexsym,xcolor}
\usepackage{graphicx}
\usepackage[mathscr]{eucal}
\usepackage{harpoon}
\usepackage{amssymb}
\usepackage[
linktocpage=true,colorlinks,citecolor=magenta,linkcolor=blue,urlcolor=magenta]{hyperref}
\allowdisplaybreaks[4]
\providecommand{\U}[1]{\protect \rule{.1in}{.1in}}

\newtheorem{theorem}{Theorem}[section]

\newtheorem{corollary}[theorem]{Corollary}
\newtheorem{definition}[theorem]{Definition}
\newtheorem{example}[theorem]{Example}
\newtheorem{Lemma}[theorem]{Lemma}
\newtheorem{Proposition}[theorem]{Proposition}
\newtheorem{Theorem}{Theorem}

\theoremstyle{remark}
\newtheorem{remark}[theorem]{Remark}

\numberwithin{equation}{section}

\def\R{\mathbb R}

\def\div{\operatorname{div}}

\def\e{\varepsilon}
\def\om{\Omega}
\def\iy{\infty}

\def\pa{\partial}

\def\rt{\rightarrow}

\def\f{\frac}
\def\bt{\begin{theorem}}
\def\et{\end{theorem}}
\def\bl{\begin{Lemma}}
\def\el{\end{Lemma}}
\def\bd{\begin{definition}}
\def\ed{\end{definition}}
\def\bc{\begin{corollary}}
\def\ec{\end{corollary}}
\def\bprop{\begin{Proposition}}
\def\eprop{\end{Proposition}}
\def\bp{\begin{proof}}
\def\ep{\end{proof}}
\def\bx{\begin{example}}
\def\ex{\end{example}}
\def\br{\begin{remark}}
\def\er{\end{remark}}
\def\be{\begin{equation}}
\def\ee{\end{equation}}
\def\bal{\begin{align}}
\def\bn{\begin{enumerate}}
\def\en{\end{enumerate}}
\def\eal{\end{align}}

\def\bg{\begin{align*}}
\def\eg{\end{align*}}
\def\bcs{\begin{cases}}
\def\ecs{\end{cases}}

\def\RNum#1{\uppercase\expandafter{\romannumeral #1\relax}}

\def\bean{\begin{eqnarray*}}
\def\eean{\end{eqnarray*}}

\def\p{\mathcal P}
\def\s{\mathcal S}

\begin{document}
\title[Weakly coupled Schr\"odinger system with large exponents]{Concentration phenomena of positive solutions to weakly coupled Schr\"odinger systems with large exponents in dimension two}
\author{Zhijie Chen}
\address{Department of Mathematical Sciences, Yau Mathematical Sciences Center,
Tsinghua University, Beijing, 100084, China}
\email{zjchen2016@tsinghua.edu.cn}
\author{Hanqing Zhao}
\address{Department of Mathematical Sciences,
Tsinghua University, Beijing, 100084, China}
\email{zhq20@mails.tsinghua.edu.cn}

\keywords{}


\begin{abstract}
We study the weakly coupled
nonlinear Schr\"odinger system
\begin{equation*}
\begin{cases}
-\Delta u_1 = \mu_1  u_1^{p} +\beta  u_1^{\frac{p-1}{2}} u_2^{\frac{p+1}{2}}\text{ in }  \Omega,\\
-\Delta u_2 = \mu_2 u_2^{p}  +\beta u_2^{\frac{p-1}{2}}u_1^{\frac{p+1}{2}}  \text{ in }  \Omega,\\
u_1,u_2>0\quad\text{in }\;\Omega;\quad
u_1=u_2=0  \quad\text { on } \;\partial\Omega,
\end{cases}
\end{equation*}
where $p>1, \mu_1, \mu_2, \beta>0$ and $\Omega$ is a smooth bounded domain in $\mathbb{R}^2$. Under the natural condition that holds automatically for all positive solutions in star-shaped domains 
\begin{align*}
    p\int_{\Omega}|\nabla u_{1,p}|^2+|\nabla u_{2,p}|^2 dx \leq C, 
\end{align*}
we give a complete description of the concentration phenomena of positive solutions $(u_{1,p},u_{2,p})$ as $p\rightarrow+\infty$, including the $L^{\infty}$-norm quantization $\|u_{k,p}\|_{L^\infty(\Omega)}\to \sqrt{e}$ for $k=1,2$, the energy quantization $p\int_{\Omega}|\nabla u_{1,p}|^2+|\nabla u_{2,p}|^2dx\to 8n\pi e $ with $n\in\mathbb{N}_{\geq 2}$, and so on. In particular, we show that the ``local mass'' contributed by each concentration point must be one of $\{(8\pi,8\pi), (8\pi,0),(0,8\pi)\}$.
\end{abstract}

\maketitle

\section{Introduction}
\label{section-1}
In this paper, we study the asymptotic behavior of positive solutions to the following weakly coupled nonlinear Schr\"odinger system
\begin{equation}\label{eq-1.1}
\begin{cases}
-\Delta u_1 = \mu_1  u_1^{p} +\beta  u_1^{\frac{p-1}{2}}u_2^{\frac{p+1}{2}} \text{ in }  \Omega,\\
-\Delta u_2 = \mu_2 u_2^{p}  +\beta u_2^{\frac{p-1}{2}}  u_1^{\frac{p+1}{2}} \text{ in }  \Omega,\\
u_1,u_2>0\quad\text{in }\;\Omega;\quad
u_1=u_2=0  \quad\text { on } \;\partial\Omega,
\end{cases}
\end{equation}
where $p>1, \mu_1, \mu_2, \beta>0$ and $\Omega$  is a smooth bounded domain in $\mathbb{R}^2$. Without loss of generality, we may always assume $\mu_1\geq \mu_2$. 
Systems of this type arise in many physical problems such as nonlinear optics and Bose-Einstein condensates, and have been widely studied in the last decades, and it seems impossible for us to list all the references. We refer the reader to \cite{B-T-W-10, D-W-W-10,GS, L-W-05,MMP, NTTV, Sirakov, SV} and references therein. For system \eqref{eq-1.1}, the existence of positive solutions can be easily obtained via variational methods. Furthermore, a direct computation shows that \eqref{eq-1.1} always has positive solutions of the form $(k_pu_p, l_pu_p)$ for $p$ large, where $u_p$ are positive solutions of the following Lane-Emden equation in dimension two,
\begin{equation}\label{eq-1.1***}
\begin{cases}
-\Delta u= u^p \,\,\text{ in }  \Omega,\\
u>0\quad\text{in }\;\Omega;\quad u=0  \quad\text { on } \partial\Omega,
\end{cases}
\end{equation}
and $k_p, l_p>0$ are solutions of
\[\mu_1k_p^{p-1}+\beta k_p^{\frac{p-3}{2}}l_p^{\frac{p+1}{2}}=1,\quad \mu_2l_p^{p-1}+\beta l_p^{\frac{p-3}{2}}k_p^{\frac{p+1}{2}}=1.\]

Our main motivation comes from the study of the concentration phenomena of positive solutions $u_p$ of the Lane-Emden equation \eqref{eq-1.1***} 
satisfying 
\begin{align}\label{eq-1-2*****}
     p\int_{\Omega}|\nabla u_p|^2 dx \leq C,\quad \text{for large} \; p,
\end{align}
as $p\rightarrow +\infty$ in a series of papers \cite{A-G-04,D-G-I-P-18,D-I-P-17,D-I-P-19,Rn-Wei-94,Ren-wei-95,T-19}. Remark that Kamburov-Sirakov \cite{K-S-18, K-S-23} proved that the assumption \eqref{eq-1-2*****} holds automatically for all positive solutions when $\Omega$ is a star-shaped domain.

Let $G(x,y)=G(y,x)$ be the Green function defined by
\[-\Delta_yG(x,y)=\delta_x\quad\text{in }\Omega;\quad G(x,y)=0\quad\text{for }y\in\partial\Omega,\]
where $\delta_x$ denotes the Dirac measure at $x\in\Omega$.
  It is well known that there exists $C>0$ such that for all $x,y\in\Omega, x\neq y$,
    \begin{align}\label{eq-2.1}
        |G(x,y)|\leq C\left(1+\log\frac{1}{|x-y|}\right),\quad |\nabla_x G(x,y)|\leq\frac{C}{|x-y|}.
    \end{align}
Let $H(x,y)$ be the regular part of $G(x,y)$ given by
\begin{align}\label{eq-1-6^*}
    G(x,y)=\frac{1}{2\pi}\log\frac1{|x-y|}-H(x,y),
\end{align}
and $R(x):=H(x,x)$ be the well-known Robin function.
We collect the following results from \cite{A-G-04,D-G-I-P-18,D-I-P-17,D-I-P-19,T-19} as follows.
\begin{Theorem}\label{th-AA}
     Let $u_p$ be a family of positive solutions of the Lane-Emden equation \eqref{eq-1.1***} satisfying \eqref{eq-1-2*****}. Then there exists a finite set of points $\mathcal{S}=\{x_1,\ldots,x_N\}\subset\Omega$ with $N\in \mathbb{N}_{\geq 1}$ such that, up to a subsequence, the following properties hold:
     \begin{itemize}
         \item [(1)]  the local $L^{\infty}$-norm quantization:
         \begin{align}
             \lim\limits_{p\rightarrow\infty}\Vert u_p\Vert_{L^{\infty}(B_r(x_i))}=\sqrt{e},
         \end{align}
         for $r>0$ small, where $B_r(x):=\{y\in\mathbb R^2 : |x-y|<r\}$. In particular, 
         $$  \lim\limits_{p\rightarrow\infty}\Vert u_p\Vert_{L^{\infty}(\Omega)}=\sqrt{e}.$$
         \item [(2)] as $p\rightarrow\infty$,
         \begin{align}
        p u_p\rightarrow 8\pi \sqrt{e} \sum_{i=1}^N  G(x_i,\cdot)\quad\text{ in }\,\,C^{2}_{loc}(\overline{\Omega}\setminus\mathcal{S}).
    \end{align}
         \item [(3)] the concentration points $\{x_i\}_{i=1}^N$ satisfy
    \begin{align}
         \nabla_x  R(x_i) -2\sum_{l\neq i }  \nabla_x G(x_i,x_l)=0,\quad\forall i.
    \end{align}
    \item [(4)] the energy quantization:
     \begin{align}
        \lim_{p\rightarrow\infty}p\int_{\Omega}|\nabla u_p|^2 dx =8\pi N e.
     \end{align}
     \item [(5)] Let $x_{p,i}\in \overline{B_r(x_i)}$ be given by
         \begin{align*}
         u_p(x_{p,i}):=\max_{x\in B_r(x_i)}u_p(x),
     \end{align*}
     and define
       \begin{align}\label{eq-1.12}
         v_{p,i}(z):=\frac{p}{u_p(x_{p,i})}\left(u_p(x_{p,i}+\mu_{p,i} z)-u_p(x_{p,i})\right),
     \end{align}
    where $\mu_{p,i}^{-2}:=pu_p(x_{p,i})^{p-1}$. Then $v_{p,i}\rightarrow U$ in $C_{loc}^2(\mathbb{R}^2)$, where $U(z)=-2\log(1+\frac{|z|^2}{8})$ is an entire solution of  the Liouville equation
     \begin{align}\label{eq-chen-li}
   -\Delta U=e^U \quad\text{in}\,\,\mathbb{R}^2,\quad
             \int_{\mathbb{R}^2} e^U dz<+\infty.
       \end{align}
     \end{itemize}
\end{Theorem}
 Theorem \ref{th-AA} gives a complete description of the asymptotic behavior of positive solutions of \eqref{eq-1.1***} as $p\to+\infty$. Recently in remarkable works \cite{DGIP, GILY}, Theorem \ref{th-AA} was applied to prove the uniqueness of positive solutions of \eqref{eq-1.1***} in convex domains for $p$ large.

The purpose of this paper is to generalize Theorem \ref{th-AA} to the weakly coupled nonlinear Schr\"odinger system \eqref{eq-1.1}  satisfying
\begin{align}\label{eq-1-2}
    p\int_{\Omega}|\nabla u_{1,p}|^2+|\nabla u_{2,p}|^2 dx \leq C, \quad \text{ for large }\; p. 
\end{align}
Our first result shows that this condition holds automatically for all positive solutions for star-shaped domains.
\begin{theorem}\label{prop-5-2}
    Let $p_0 > 1$. Then there exists a constant $C$, depending only on $p_0$ and $\Omega$, such that for any $p\geq p_0$ and
any positive solution $(u_{1,p},u_{2,p})$ of \eqref{eq-1.1}, there holds
$$\|u_{1,p}\|_{L^\infty(\Omega)}+\|u_{2,p}\|_{L^\infty(\Omega)}\leq C.$$
If in addition that $\Omega$ is star-shaped, then there exists a constant $C_1$, depending only on $p_0$ and $\Omega$, such that for any $p\geq p_0$ and
any positive solution $(u_{1,p},u_{2,p})$, there holds
\begin{align*}
  p\int_{\Omega} |\nabla u_{1,p}|^2+|\nabla u_{2,p}|^2 dx\leq C_1.
\end{align*}
\end{theorem}
As mentioned earlier, this result was proved by  Kamburov-Sirakov for the Lane-Emden equation \eqref{eq-1.1***} and strictly star-shaped domains in \cite{K-S-18}, and for the Lane-Emden system and star-shaped domains in \cite{K-S-23}. Here we generalize their result to the weakly coupled system \eqref{eq-1.1}. We will follow the arguments in \cite{K-S-18, K-S-23} to prove Theorem \ref{prop-5-2} in Section 2.

Our next results generalize Theorem \ref{th-AA}, which is the first step for us to consider further the uniqueness problem of positive solutions for system \eqref{eq-1.1} in convex domains in the future.
\begin{theorem}\label{Th-1LL}
    Let $(u_{1,p},u_{2,p})$ be a family of positive solutions to system \eqref{eq-1.1} satisfying \eqref{eq-1-2}.
Then there exist $N\in \mathbb{N}_{\geq 1}$ and a finite set of concentration points $\mathcal{S}=\{x_1,\ldots,x_N\}\subset\Omega$, such that, up to a subsequence, the following properties hold:
\begin{itemize}
    \item [(1)] the local $L^{\infty}$-norm quantization: there exist $m_{k,i}\in\{0,1\}$ for $k=1,2$ and $i=1,\cdots,N$ such that 
    \begin{align*}
     \lim\limits_{p\rightarrow\infty}\Vert u_{1,p}\Vert_{L^{\infty}(B_r(x_i))}=m_{1,i}\sqrt{e},&\quad\lim\limits_{p\rightarrow\infty}\Vert u_{2,p}\Vert_{L^{\infty}(B_r(x_i))}=m_{2,i}\sqrt{e},
    \end{align*}
    for $r>0$ small.
    Moreover, $\max\{m_{1,i}, m_{2,i}\}=1$ for all $i$ and
    \begin{align*}
    \sum_{i=1}^N m_{1,i}\geq 1,&\quad\sum_{i=1}^N m_{2,i}\geq 1.
    \end{align*}
    As a consequence,
    \begin{align*}
        \lim\limits_{p\rightarrow\infty}\Vert u_{1,p}\Vert_{L^\infty(\Omega)}=\lim\limits_{p\rightarrow\infty}\Vert u_{2,p}\Vert_{L^\infty(\Omega)}=\sqrt{e}.
    \end{align*}     
    \item [(2)] Define 
    \begin{align*}
        \mathcal{S}_k:=\{x_i\in\mathcal{S}: m_{k,i}=1\}\neq \emptyset,\quad \text{for }k=1,2.
    \end{align*}
   Then $\mathcal{S}_1\cup\mathcal{S}_2=\mathcal S$ and as $p\rightarrow\infty$, 
    \begin{align*}
        p u_{1,p}\rightarrow 8\pi \sqrt{e} \sum_{x_i\in\s_1} G(x_i,\cdot)\quad\text{ in }\,\,C^{2}_{loc}(\overline{\Omega}\setminus\mathcal{S}_1),\\
        p u_{2,p}\rightarrow 8\pi \sqrt{e} \sum_{x_i\in\s_2} G(x_i,\cdot)\quad\text{ in }\,\,C^{2}_{loc}(\overline{\Omega}\setminus\mathcal{S}_2).
    \end{align*}

    \item [(3)]  the concentration points $\{x_i\}_{i=1}^N$ satisfy
    \begin{align*}
        \sum_{k=1}^2 m_{k,i}\left(m_{k,i} \nabla_x  R(x_i) -2\sum_{l\neq i, l=1}^N m_{k,l} \nabla_x G(x_i,x_l)\right)=0,\quad\forall i.
    \end{align*}
    \item [(4)]
     the energy quantization:
     \begin{align*}
        \lim_{p\rightarrow\infty}p\int_{\Omega}|\nabla u_{1,p}|^2+|\nabla u_{2,p}|^2 dx =8\pi e\sum_{i=1}^N (m_{1,i}+m_{2,i}).
     \end{align*}
\end{itemize}
\end{theorem}

To describe the local profile of $(u_{1,p}, u_{2,p})$ near a concentration point, we define
\[M_p:=\max\{u_{1,p}, u_{2,p}\}.\]
For small $r>0$ and each $i\in [1,N]$, we let $y_{p,i}\in\overline{B_{r}(x_i)}$ such that 
\begin{align*}
        M_p(y_{p,i}) :=\max_{x\in B_r(x_{i})}M_p(x).
    \end{align*}
We will prove in Section \ref{section-3} that $y_{p,i}\to x_i$ as $p\to\infty$. Define
 $\varepsilon_{p,i}^{-2}:=p M_p(y_{p,i})^{p-1}$
 and similarly as Theorem  \ref{th-AA}, we consider the scaling
for $k=1,2$:
    \begin{align}
         v_{k,p,i}^*(z):=p\frac{u_{k,p}(\varepsilon_{p,i}z+y_{p,i})-M_p(y_{p,i})}{M_p(y_{p,i})}.
        \end{align} 
The following results play key roles in the proof of Theorem \ref{Th-1LL}.
\begin{theorem}\label{th-1---1}
    For each $x_i\in \s$ obtained in Theorem \ref{Th-1LL}, up to a subsequence, one of the following alternatives happens. 
    \begin{itemize}
        \item [(1)] $(v_{1,p,i}^*,v_{2,p,i}^*)\rightarrow (V_1,V_2)$ in $C^{2}_{loc}(\mathbb{R}^2)\times C^{2}_{loc}(\mathbb{R}^2)$,
        where $(V_1, V_2)$ is given by
        \begin{align}
            V_1(z)+\log\frac{\mu_1}{\mu_2}=V_2(z)=-2\log\left({1+\frac{\mu_2}{8}(1+\frac{\beta}{\sqrt{\mu_1\mu_2}}) |z|^2}\right),
        \end{align}
       and solves 
       the Liouville-type system 
 \begin{equation}
\begin{cases}\label{l-sys}
-\Delta U_1 = \mu_1 e^{U_1} +\beta e^{\frac{U_1+U_2}{2}}  \quad\text{in} \, \mathbb{R}^2,\\
-\Delta U_2 = \mu_2 e^{U_2} +\beta e^{\frac{U_1+U_2}{2}}  \quad\text{in} \, \mathbb{R}^2,\\
\int_{\mathbb{R}^2}e^{U_1} < +\infty \,\, ,\,\,\int_{\mathbb{R}^2}e^{U_2} < +\infty.
\end{cases}
\end{equation} 
        \item [(2)] For some $k\in\{1,2\}$,
         $v_{k,p,i}^*\rightarrow \widetilde{V}_k$ in $C^{2}_{loc}(\mathbb{R}^2)$, where 
         \begin{align}
             \widetilde{V}_k(z)=-2\log\left({1+\frac{\mu_{k}}{8}|z|^2}\right)
         \end{align}
        solves the Liouville equation 
        \begin{align}\label{fc-Liouville}
                -U_k=\mu_k e^{U_k} \quad\text{in} \, \mathbb{R}^2,\quad
                \int_{\mathbb{R}^2} e^U<+\infty,
        \end{align}
        and $v_{3-k,p,i}^*\rightarrow -\infty$ uniformly in  any compact subsets of $\mathbb{R}^2$.  
        \end{itemize}
\end{theorem}

\begin{theorem}\label{thm-localmass}
For each $x_i\in \s$ obtained in Theorem \ref{Th-1LL} and small $r>0$, we define the ``local mass'' contributed by this concentration point $x_i$:
\begin{align*}
    \sigma_{k,i}:=\lim\limits_{p\rightarrow \infty}  \frac{p}{M_p(y_{p,i})} \int_{B_{r}(x_{i})} \left(\mu_k u_{k,p}^p +\beta u_{k,p}^{\frac{p-1}{2}}u_{3-k,p}^{\frac{p+1}{2}}\right) dx,\quad k=1,2.
\end{align*}
Then $(\sigma_{1,i}, \sigma_{2,i})\in \{(8\pi, 8\pi), (8\pi, 0), (0, 8\pi)\}$.
\end{theorem}

\begin{remark}\
    \begin{itemize}
    \item[(1)]  Note that the weakly coupled system \eqref{eq-1.1} always has solutions of the form $(u, 0)$ and $(0, u)$, which brings additional difficulties compared to the single equation case as stated in Theorem \ref{th-AA}. Theorem \ref{thm-localmass} proves that the local mass has three possibilities  $(8\pi, 8\pi), (8\pi, 0), (0, 8\pi)$, and we believe that each of them would happen, which will be studied elsewhere. It is interesting to point out that in the bubbling analysis of the $SU(3)$ Toda system, it was proved by \cite{Jost-Lin-Wang-06} that the local mass must be one of $(8\pi, 8\pi), (8\pi, 4\pi)$,  $(4\pi, 8\pi)$,  $(4\pi,0), (0,4\pi)$, and it was proved by \cite{MW-JDE} that each of these five possibilities really happens.
    \item[(2)] Comparing to the single equation case, a new situation is that besides the Liouville equation \eqref{fc-Liouville}, the Liouville-type system \eqref{l-sys} also appears as a limiting equation after a suitable scaling near a concentration point. Therefore, the classification of solutions of this system \eqref{l-sys} plays a crucial role in our proof of Theorems \ref{Th-1LL} and \ref{thm-localmass}. In our previous paper \cite{our-24}, we succeeded to give the classification of solutions for this system \eqref{l-sys}; see Theorem \ref{th-3.1} in Section \ref{section-3} for the precise statement.

        \item [(3)] Comparing to the single equation case, the weakly coupled system \eqref{eq-1.1} is more difficult to deal with, and our proofs of Theorems \ref{Th-1LL}-\ref{thm-localmass} need to combine those ideas from \cite{D-G-I-P-18, Franc-Isabella-Filomena-15,D-I-P-17,Jost-Lin-Wang-06,L-W-Z-13,L-Z-13}, where the bubbling behaviors were studied for the Lane-Emden equation \eqref{eq-1.1***}  (see \cite{D-G-I-P-18, Franc-Isabella-Filomena-15, D-I-P-17}), the $SU(3)$ Toda system (see \cite{Jost-Lin-Wang-06,L-W-Z-13}) and the Liouville system (see \cite{L-Z-13}), respectively.
         \end{itemize}
\end{remark}

The rest of this paper is organized as follows. 
In Section \ref{section-A}, we study the a priori estimate and prove Theorem \ref{prop-5-2}. In Section \ref{section-3}, we study the existence of the concentration set $\s$ and prove Theorem \ref{th-1---1}.
In Section \ref{section-4}, we give the local estimates around each concentration point and prove Theorem \ref{thm-localmass}.
In Section \ref{section-5}, we complete the proof of Theorem \ref{Th-1LL}.

\textbf{Notations}.
Throughout the paper,  we always use $C, C_0, C_1,\cdots$ to denote constants that are independent of $p$ (possibly different in different places). Conventionally, we use $o(1)$ to denote quantities that converge to $0$ (resp. use $O(1)$ to denote quantities that remain bounded) as $p\to\infty$. Denote $B_r(x):=\{y\in\R^2 : |y-x|<r\}$ and $B_r:=B_r(0)$. For any $s>1$, we denote $\|u\|_s:=(\int_{\Omega}|u|^sdx)^{1/s}$.

\section{Uniform a priori estimate}
\label{section-A}
In this section, we follow the arguments from \cite{K-S-18, K-S-23} to prove Theorem \ref{prop-5-2}. Let $p_0>1$ and $(u_{1,p}, u_{2,p})$ be any positive solution of \eqref{eq-1.1}
for $p\geq p_0$.

\begin{Lemma}\label{lemma2-1}
   There are positive constants $\delta$ depending only on $\Omega$, and $C$ on depending only $p_0$ and $\Omega$, such that for $p\geq p_0$ and $k=1,2$,
\begin{itemize}
    \item [(1)] $dist (x_{k,p},\partial\Omega)\geq \delta$, where $x_{k,p}$ is a maximum point of $u_{k,p}$, i.e. $u_{k,p}(x_{k,p})=\max_{x\in\Omega} u_{k,p}(x)$.
    \item [(2)] 
   $ \int_{\Omega} u_{k,p}^p dx \leq C$ and
   $ \int_{\Omega} u_{k,p}^p dx \leq C\int_{\Omega} u_{k,p} dx$.
   
\end{itemize}
\end{Lemma}
\begin{proof}
Like \cite[Proposition 2.1]{K-S-18, K-S-23}, 
    by using the moving planes technique together with the Kelvin transform as in \cite[pp. 45-52]{DLN},
    we can prove the existence of positive constants $\delta$ and $\gamma$ depending only on $\Omega$, such that the following holds:
    \begin{itemize}
 \item   For each $x\in\Omega_{2\delta}:=\{y\in\Omega: dist(y,\partial\Omega)<2\delta\}$, there exists a measurable set $I_x$ such that
    \begin{itemize}
        \item[(1)] $I_x\subset \Omega\setminus\Omega_{\delta}=\{y\in\Omega: dist(y,\partial\Omega)\geq\delta\}$,
        \item[(2)] $|I_x|\geq \gamma$, where $|I_x|$ denotes the Lebesgue measure of $I_x$,
        \item[(3)]  $u_{k,p}(x)\leq u_{k,p}(\xi)$ for all $\xi\in I_x $ and $k=1,2$.
    \end{itemize}
\end{itemize}
    For instance, when $\Omega$ is convex, $I_x$ can be taken to be a part of a cone with a vertex at $x$, while non-convex domains can be treated by the Kelvin inversion of neighborhoods of boundary points where the boundary is not convex. In particular, this implies the assertion (1).
    
  To prove the assertion (2),  we denote $\psi$ to be the first eigenfunction of $-\Delta$ in $H_0^1(\Omega)$ with the first eigenvalue $\lambda_1=\lambda_1(\Omega)$:
    \begin{align}\label{fc-lambda1}
        \begin{cases}
            -\Delta \psi=\lambda_1 \psi \quad\text{in}\quad\Omega,\\
           \psi>0\quad\text{in }\Omega,\quad \psi=0\quad\text{on}\;\partial\Omega,
        \end{cases} \text{normalized such that}\;\|\psi\|_{\infty}=1.
    \end{align}
   It follows that for $k\in\{1,2\}$,
    \begin{align}\label{fc2-2}
        \lambda_1 \int_{\Omega} u_{k,p} \psi ~dx=\int_{\Omega}( \mu_k u_{k,p}^{p}+\beta u_{k,p}^{\frac{p-1}{2}}u_{3-k,p}^{\frac{p+1}{2}})\psi ~dx\geq \mu_k\int_{\Omega} u_{k,p}^{p} \psi dx.
    \end{align}
    Since $$\mu_k u_{k,p}^{p}\geq 2\lambda_1 u_{k,p}\quad \text{if}\quad u_{k,p}\geq t_{0}:=\sup_{p\geq p_0}\left(\frac{2\lambda_1}{\min\{\mu_1,\mu_2\}}\right)^{\frac{1}{p-1}},$$ we have 
    \begin{align*}
        &2\lambda_1 \int_{\Omega} (u_{1,p}+u_{2,p})\psi ~dx\nonumber\\
        =&\sum_{k=1}^2 2\lambda_1\int_{\{u_{k,p}>t_{0}\}}u_{k,p}\psi ~dx+ \sum_{k=1}^2 2\lambda_1\int_{\{u_{k,p}\leq t_{0}\}}u_{k,p}\psi ~dx\\
        \leq &\int_{\Omega} (\mu_1 u_{1,p}^p+\mu_2 u_{2,p}^p)\psi~ dx+ 4\lambda_1 t_0|\Omega|\nonumber\\
        \leq & \lambda_1 \int_{\Omega} (u_{1,p}+u_{2,p})\psi ~dx+ 4\lambda_1 t_0|\Omega|,\nonumber
    \end{align*}
    which implies
    $$ \int_{\Omega} (u_{1,p}+u_{2,p})\psi ~dx\leq C.$$
    From here and \eqref{fc2-2},  we obtain
    \begin{align}\label{eq-5-17}
        \int_{\Omega} u_{k,p}^p \psi~dx\leq \frac{\lambda_1}{\mu_k}\int_{\Omega} u_{k,p}\psi ~dx\leq  C.
    \end{align}
    Note
    $$c_0:= \inf_{\Omega\setminus\Omega_{\delta}} \psi>0.$$
    For any $y\in\Omega_{2\delta}$, it follows from the properties (1)-(3) of $I_y\subset \Omega\setminus\Omega_{\delta}$ that
    \begin{align*}
        \int_{\Omega} u_{k,p}^p \psi~dx\geq \int_{I_y} u_{k,p}^p \psi~dx\geq \gamma u_{k,p}^p(y)  \inf_{I_y} \psi \geq \gamma c_0 u_{k,p}^p(y).
    \end{align*}
    Therefore,
    \begin{align*}
    \int_{\Omega} u_{k,p}^p dx &\leq \int_{\Omega_{2\delta}} u_{k,p}^p dx +\int_{\Omega\setminus\Omega_{2\delta}} u_{k,p}^p dx\\
    &\leq  \frac{|\Omega|}{\gamma c_0} \int_{\Omega} u_{k,p}^p \psi~dx+\frac1{c_0}\int_{\Omega\setminus\Omega_{2\delta}} u_{k,p}^p\psi dx\leq C\int_{\Omega} u_{k,p}^p \psi~dx.
    \end{align*}
From here and \eqref{eq-5-17}, we obtain $ \int_{\Omega} u_{k,p}^p dx \leq C$ and
   $ \int_{\Omega} u_{k,p}^p dx \leq C\int_{\Omega} u_{k,p} dx$.
\end{proof}

\begin{Lemma}\label{lm-5-4}
 There exists a constant $C$, depending only on $p_0$ and $\Omega$, such that for $p\geq p_0$
and $k=1,2$, 
\begin{align}
      \Vert u_{k,p}\Vert_{L^{\infty}(\Omega)} \leq C. 
\end{align}
\end{Lemma}
\begin{proof} 
Without loss of generality, we may assume
\begin{align}
    \mathcal{M}_p:=u_{1,p}(x_{1,p})=\max_{\Omega}u_{1,p}\geq \max_{\Omega}u_{2,p}.
\end{align}
Then Lemma \ref{lemma2-1} implies $dist(x_{1,p},\partial\Omega)\geq\delta$. 
Without loss of generality, we may assume $x_{1,p}=0$ and so $B_\delta\subset\Omega$.  Furthermore, by a scaling argument as \cite[Theorem 1.1]{K-S-18}, we may also assume $\Omega\subset B_1$.

Let $G(y)$ be the Green function for $-\Delta$ in $\Omega$ with the Dirichlet boundary condition and the pole at the origin. Then
$$G(y)=\frac{1}{2\pi}\log\frac{1}{|y|}-g(y),$$
where $g$ is harmonic in $\Omega$ with $g(y)=\frac{1}{2\pi}\log\frac{1}{|y|}\in [0, \frac{1}{2\pi}\log \frac1\delta]$ for $y\in \partial\Omega$, so $0\leq g(y)\leq \frac{1}{2\pi}\log \frac1\delta$ for all $y\in \Omega$. Noting from Lemma \ref{lemma2-1} and H\"{o}lder inequality that $ \int_{\Omega} u_{1,p}^{\frac{p-1}{2}} u_{2,p}^{\frac{p+1}{2}}  dx \leq C$, we see from
  the Green‘s representation formula that
    \begin{align}\label{eq-5-25}
      \mathcal{M}_p= u_{1,p}(0)=&\frac{1}{2\pi}\int_{\Omega}\log(\frac1{|y|}) ( \mu_1 u_{1,p}^{p}+\beta u_{1,p}^{\frac{p-1}{2}}u_{2,p}^{\frac{p+1}{2}})dy\\
        &-\int_{\Omega}g(y)( \mu_1 u_{1,p}^{p}+\beta u_{1,p}^{\frac{p-1}{2}}u_{2,p}^{\frac{p+1}{2}})dy\nonumber\\
        \geq &\frac{\mu_1}{2\pi}\int_{\Omega}\log(\frac1{|y|}) u_{1,p}^{p} dy-C.\nonumber
    \end{align}
 Denote
    \begin{align}
        v_{k,p}(x):=1-\frac{u_{k,p}(\mathcal{M}_p^{-\frac{p-1}{2}}x)}{\mathcal{M}_p},
    \end{align}
   then
    \begin{align*}
        \begin{cases}
            \Delta v_{1,p}=\mu_1(1-v_{1,p})^p+\beta (1-v_{1,p})^{\frac{p-1}{2}}(1-v_{2,p})^{\frac{p+1}{2}}\;\text{in}\; \widetilde{\Omega}:=\mathcal{M}_p^{\frac{p-1}{2}}\Omega\subset B_{\mathcal{M}_p^{(p-1)/2}},\\
            v_{1,p}(0)=0, \quad 0\leq v_{1,p}\leq 1,
           \quad 0\leq v_{2,p}\leq 1.
        \end{cases}
    \end{align*}
    Note from \eqref{eq-5-25} that
    \[\int_{\widetilde\Omega}\log(\mathcal{M}_p^{\frac{p-1}{2}}/|x|)(1-v_{1,p})^pdx\leq \frac{2\pi}{\mu_1}+\frac{2\pi C}{\mu_1 \mathcal{M}_p}.\]

    If $\mathcal{M}_p\leq 1$, we get the desired conclusion. So we assume that $\mathcal{M}_p> 1$, which implies $B_{\delta}\subset\Omega\subset\widetilde\Omega$. Since 
    \begin{align}
        0\leq \mu_1(1-v_{1,p})^p+\beta (1-v_{1,p})^{\frac{p-1}{2}}(1-v_{2,p})^{\frac{p+1}{2}}\leq \mu_1+\beta,
    \end{align}
    it follows from $v_{1,p}(0)=0$ and the inhomogeneous Harnack inequality \cite[Theorem 4.17]{Han-lin-book} that
    $$\|v_{1,p}\|_{L^\infty(B_r)}\leq Cr^2,\quad \forall r\in (0,\delta), $$
    and so
    \begin{align}
     v_{1,p}\leq \frac{C}{p},\quad\text{in}\quad B_{\frac{\delta}{\sqrt{p}}},
    \end{align}which implies
    \begin{align}
         &(1-v_{1,p})^p
         \geq\left(1-\frac{C}{p}\right)^p\geq e^{-C},\quad\text{in}\quad B_{\frac{\delta}{\sqrt{p}}}.
    \end{align}
    It follows that
    \begin{align}\label{eq-5-31}
    \frac{2\pi}{\mu_1}+\frac{2\pi C}{\mu_1}
   \geq &\frac{2\pi}{\mu_1}+\frac{2\pi C}{\mu_1\mathcal{M}_p}\geq\int_{B_{\frac{\delta}{\sqrt{p}}}}\log(\mathcal{M}_p^{\frac{p-1}{2}}/|x|)(1-v_{1,p})^pdx\nonumber\\
    \geq &C\int^{\frac{\delta}{\sqrt{p}}}_0
    \log\left({\mathcal{M}_p^{\frac{p-1}{2}}}/{r}\right) r dr \geq \frac{C}{p}\log\left({\sqrt{p}\mathcal{M}_p^{\frac{p-1}{2}}}/{\delta}\right) \nonumber\\
    \geq &C\left(\frac{\log p}{p}+\frac{p-1}{p}\log{\mathcal{M}_p}\right)\geq C\log{\mathcal{M}_p}\nonumber,
    \end{align}
    where we have used that for any $\alpha,\rho>0$,
    \begin{align}
        \int_0^\rho \log(\alpha/r)r dr\geq \frac{1}{2} \log(\alpha/\rho)\rho^2. 
    \end{align}
  This implies  ${\mathcal{M}_p}\leq C$.
\end{proof}

For later usage, we give the following Pohozaev identities.

\begin{Lemma}\label{lemma-2-4}
    Let $(u_{1,p},u_{2,p})$ be a solution of the system \eqref{eq-1.1}, $\Omega'\subset\Omega$ and $y\in\mathbb R^2$. Then 
    \begin{align}
        &\frac{2}{p+1}\int_{\Omega'}\left(\mu_1 u_{1,p}^{p+1}+\mu_2 u_{2,p}^{p+1}+2\beta u_{1,p}^{\frac{p+1}{2}}u_{2,p}^{\frac{p+1}{2}}\right) dx\label{eq-2--7}\\
        =&\sum_{k=1}^2\int_{\partial\Omega'}\langle\nabla u_{k,p},x-y\rangle \langle\nabla u_{k,p}, \Vec{n}(x)\rangle dS_x\nonumber\\
        &-\frac{1}{2}\sum_{k=1}^2\int_{\partial\Omega'}|\nabla u_{k,p}|^2 \langle x-y, \Vec{n}(x)\rangle dS_x\nonumber\\
        &+\frac{1}{p+1}\int_{\partial\Omega'}\langle x-y, \Vec{n}(x)\rangle\left(\mu_1 u_{1,p}^{p+1}+\mu_2 u_{2,p}^{p+1}+2\beta u_{1,p}^{\frac{p+1}{2}}u_{2,p}^{\frac{p+1}{2}}\right)dS_x\nonumber,
        \end{align}
        and
        \begin{align}
        &\frac{1}{2}\sum_{k=1}^2\int_{\partial\Omega'}|\nabla u_{k,p}|^2 n_i dS_x\label{eq-2--8} -\sum_{k=1}^2\int_{\partial\Omega'}\partial_i u_{k,p} \langle\nabla u_{k,p}, \Vec{n}(x)\rangle dS_x\\   
        =&\frac{1}{p+1}\int_{\partial\Omega'}\left(\mu_1 u_{1,p}^{p+1}+\mu_2 u_{2,p}^{p+1}+2\beta u_{1,p}^{\frac{p+1}{2}}u_{2,p}^{\frac{p+1}{2}}\right) n_i dS_x,\nonumber
    \end{align}
  where $\Vec{n}=(n_1,n_2)$ denotes the outward normal vector of $\partial\Omega'$.
\end{Lemma}
\begin{proof}
    By direct computations, we have for $k=1,2$,
    \begin{align*}
        &\langle x-y, \nabla  u_{k,p}  \rangle \Delta u_{k,p} =\div \left(\langle x-y, \nabla  u_{k,p}  \rangle \nabla u_{k,p}-\frac{|\nabla u_{k,p}|^2}{2}(x-y)\right),\\
        &\langle x-y, \nabla  u_{1,p} \rangle \left(\mu_1 u_{1,p}^p +\beta u_{1,p}^{\frac{p-1}{2}}u_{2,p}^{\frac{p+1}{2}}\right)+\langle x-y, \nabla  u_{2,p} \rangle \left(\mu_2 u_{2,p}^p +\beta u_{2,p}^{\frac{p-1}{2}}u_{1,p}^{\frac{p+1}{2}}\right)\\
        =&\frac{1}{p+1} \div \left( \left(\mu_1 u_{1,p}^{p+1}+\mu_2 u_{2,p}^{p+1}+2\beta u_{1,p}^{\frac{p+1}{2}}u_{2,p}^{\frac{p+1}{2}}\right)(x-y)\right)\nonumber\\
        &-\frac{2}{p+1}\left(\mu_1 u_{1,p}^{p+1}+\mu_2 u_{2,p}^{p+1}+2\beta u_{1,p}^{\frac{p+1}{2}}u_{2,p}^{\frac{p+1}{2}}\right). \nonumber
    \end{align*}
    Then multiplying the $k$-th equation of \eqref{eq-1.1} with $\langle x-y, \nabla  u_{k,p}  \rangle$ and integrating over $\Omega'$, we obtain \eqref{eq-2--7}. Similarly, we have
    \begin{align*}
        &\partial_i u_{k,p} \Delta u_{k,p}=\div (\partial_i u_{k,p} \nabla u_{k,p})-\frac{1}{2} \partial_i(|\nabla u_{k,p}|^2),\\
        &\partial_i u_{1,p} \left(\mu_1 u_{1,p}^p +\beta u_{1,p}^{\frac{p-1}{2}}u_{2,p}^{\frac{p+1}{2}}\right)+\partial_i u_{2,p} \left(\mu_2 u_{2,p}^p +\beta u_{2,p}^{\frac{p-1}{2}}u_{1,p}^{\frac{p+1}{2}}\right)\\
        =&\frac{1}{p+1}\partial_i\left(\mu_1 u_{1,p}^{p+1}+\mu_2 u_{2,p}^{p+1}+2\beta u_{1,p}^{\frac{p+1}{2}}u_{2,p}^{\frac{p+1}{2}}\right).\nonumber
    \end{align*}
    Then multiplying the $k$-th equation of \eqref{eq-1.1} with $\partial_i u_{k,p}$ and integrating over $\Omega'$, we obtain \eqref{eq-2--8}. 
\end{proof}

We also need the following estimate for superharmonic functions from \cite{Sirakov2022}.

\begin{theorem} \cite[Theorem 1.3]{Sirakov2022}\label{thm-Sirakov}
Assume $u\in H_0^1(\Omega)$ is a nonnegative weak solution to $-\Delta u\geq 0$ in a bounded $C^{1,1}$ domain $\Omega\subset\mathbb R^N$. Then for each $t<\frac{N}{N-1}$,
\[\inf_{\Omega}\frac{u(x)}{dist(x,\partial\Omega)}\geq C\|u\|_{L^t(\Omega)},\]
where the constant $C>0$ depends on $\Omega$, $t$ and $N$.
\end{theorem}
Now we are ready to prove Theorem \ref{prop-5-2}.
\begin{proof}[Proof of Theorem \ref{prop-5-2}]
Assume in addition that the smooth bounded domain $\Omega$ is star-shaped. Without loss of generality, we may assume that $\Omega$ is star-shaped with respect to the origin, i.e.
\begin{align}\label{eq-5-34}
    \langle x,\Vec{n}(x)\rangle\geq 0,\quad\forall x\in\partial\Omega.
\end{align}
  Applying \eqref{eq-2--7} with $\Omega'=\Omega$ and $y=0$, it follows from $\nabla u_{k,p}=-|\nabla u_{k,p}|\vec{n}$ on $\partial\Omega$ that
  \begin{align}\label{eq-5-35}
      &\frac{4}{p+1}\int_{\Omega}\left(\mu_1 u_{1,p}^{p+1}+\mu_2 u_{2,p}^{p+1}+2\beta u_{1,p}^{\frac{p+1}{2}}u_{2,p}^{\frac{p+1}{2}}\right) dx\nonumber\\
        =&\int_{\partial\Omega}\left(|\nabla u_{1,p}|^2 \langle x, \Vec{n}(x)\rangle+|\nabla u_{2,p}|^2 \langle x, \Vec{n}(x)\rangle\right)dS_x.
  \end{align}
  Using Theorem \ref{thm-Sirakov} with $t=1$, we get that for any $x\in\partial\Omega$,
  \begin{equation*}
  |\nabla u_{k,p}(x)|=-\langle \nabla u_{k,p}(x),\vec{n}(x)\rangle\geq \inf_{\Omega}\frac{u_{k,p}(\cdot)}{dist(\cdot,\partial\Omega)}\geq C\|u_{k,p}\|_{L^1(\Omega)}.
  \end{equation*}
  From here, \eqref{eq-5-34} and Lemma \ref{lemma2-1}, we get
  \begin{align*}
  \int_{\partial\Omega}|\nabla u_{k,p}|^2 \langle x, \Vec{n}(x)\rangle dS_x&\geq C^2\|u_{k,p}\|_{L^1(\Omega)}^2\int_{\partial\Omega}\langle x, \Vec{n}(x)\rangle dS_x\\
  &=C^2\|u_{k,p}\|_{L^1(\Omega)}^2\int_{\Omega}{\rm div}(x)dx=C_1\|u_{k,p}\|_{L^1(\Omega)}^2\\
  &\geq C_2\left(\int_{\Omega}u_{k,p}^{p}dx\right)^2.
  \end{align*}
  This, together with Lemma \ref{lm-5-4} and \eqref{eq-5-35}, implies
  \begin{align*}
  \frac{C_2}{2}\left(\sum_{k=1}^2\int_{\Omega}u_{k,p}^{p}dx\right)^2&\leq   C_2\sum_{k=1}^2\left(\int_{\Omega}u_{k,p}^{p}dx\right)^2\leq \sum_{k=1}^2 \int_{\partial\Omega}|\nabla u_{k,p}|^2 \langle x, \Vec{n}(x)\rangle dS_x\\
 & =\frac{4}{p+1}\int_{\Omega}\left(\mu_1 u_{1,p}^{p+1}+\mu_2 u_{2,p}^{p+1}+2\beta u_{1,p}^{\frac{p+1}{2}}u_{2,p}^{\frac{p+1}{2}}\right) dx\\
 &\leq \frac{C_3}{p}\sum_{k=1}^2\int_{\Omega}u_{k,p}^{p}dx,
  \end{align*}
so
\begin{align}\label{eq-5-43}
    p\int_{\Omega} \left( u_{1,p}^p+  u_{2,p}^p \right)dx\leq C.
\end{align}
Consequently,
\begin{align}
    &p\int_{\Omega} |\nabla u_{1,p}|^2+|\nabla u_{2,p}|^2 dx\nonumber\\
    =&p\int_{\om} \left(\mu_1 u_{1,p}^{p+1}+\mu_2 u_{2,p}^{p+1}+2\beta u_{1,p}^{\frac{p+1}{2}}u_{2,p}^{\frac{p+1}{2}} \right) dx\\
    \leq & C p \int_{\Omega} \left( u_{1,p}^p+ u_{2,p}^p \right)dx \leq C.\nonumber
\end{align}
The proof is complete.
\end{proof}

\section{Analysis of the concentration set}
\label{section-3}

First, we recall some known results that will be used in our following arguments.

\begin{Lemma}\cite[Lemma 2.1]{Rn-Wei-94}\label{lm2.4}
Let $D\subset \mathbb{R}^2$ be a smooth bounded domain. Then for every $p>1$ there exists $S_p>0$ such that 
\begin{align}
    \left\Vert v \right\Vert_{L^{p+1}(D)}\leq S_p (p+1)^{\frac{1}{2}}\left\Vert \nabla v \right\Vert_{L^{2}(D)},\quad \forall v\in H_0^1 (D).
\end{align}
Moreover,
\begin{align}
    \lim\limits_{p\rightarrow +\infty}S_p={(8\pi e)^{-\frac{1}{2}}}.
\end{align}
\end{Lemma}

\begin{theorem} \cite[Theorem 1.1]{our-24}\label{th-3.1}
   Let $U_1, U_2 \in L^1_{loc}(\mathbb{R}^2)$ and $(U_1,U_2)$ be a solution of the Liouville-type system 
   \begin{equation}\label{eq-C-Z}
\begin{cases}
-\Delta U_1 = \mu_1 e^{U_1} +\beta e^{\frac{U_1+U_2}{2}}  \quad\text{in} \, \mathbb{R}^2,\\
-\Delta U_2 = \mu_2 e^{U_2} +\beta e^{\frac{U_1+U_2}{2}}  \quad\text{in} \, \mathbb{R}^2,\\
\int_{\mathbb{R}^2}e^{U_1} < \infty \,\, ,\,\,\int_{\mathbb{R}^2}e^{U_2} < \infty,
\end{cases}
\end{equation} 
in the sense of distribution.
Then there exist $\lambda\in(0,+\infty)$ and $x_0\in\mathbb{R}^2$ such that
    \begin{align*}
    U_k(x)=\log\frac{8\lambda^2}{(1+\lambda^2|x-x_0|^2)^2}-\log\left(1+\frac{\beta}{\sqrt{\mu_1\mu_2}}\right)-\log \mu_k,\; k=1,2.
    \end{align*}
Consequently, $$\int_{\mathbb{R}^2}\mu_{k} e^{U_k} +\beta e^{\frac{U_1+U_2}{2}}dx=8\pi,\quad k=1,2,$$
or equivalently,
$$\int_{\mathbb R^2}e^{U_k}dx=\frac{8\pi}{\mu_k(1+\frac{\beta}{\sqrt{\mu_{1}\mu_{2}}})},\quad k=1,2.$$
\end{theorem}

From now on, we let $p\geq p_0$ for some $p_0>1$, and let $(u_{1,p},u_{2,p})_{p\geq p_0}$ be a family of positive solutions of \eqref{eq-1.1} satisfying the natural condition \eqref{eq-1-2}, i.e.
\begin{equation}\label{eq-1-2-3}\sup_{p\geq p_0}p\int_{\Omega} |\nabla u_{1,p}|^2+|\nabla u_{2,p}|^2 dx\leq C.\end{equation}
 In this section, we study the bubbling phenomenon of
 $(pu_{1,p}, pu_{2,p})$  as $p\to +\infty$
and prove Theorem \ref{th-1---1}.
Denote
\begin{align}\label{eq-4.4}
    M_p(x):=\max\left\{u_{1,p}(x),u_{2,p}(x)\right\}
\end{align}
and let $x_{p,1}\in\Omega$ such that
\begin{align}\label{eq-4.5}
    M_p(x_{p,1})=\max\limits_{x\in\overline{\Omega}} M_{p}(x).
\end{align}
Then it follows from Lemma \ref{lm-5-4} that
    \begin{align}\label{eq-4.8}
    M_p(x_{p,1})=\Vert M_p\Vert_{\infty}=\max\{\Vert u_{1,p}\Vert_{\infty}, \Vert u_{2,p}\Vert_{\infty}\}\leq C,\quad\forall p\geq p_0.
    \end{align}
    
    \begin{Proposition}\label{prop-4.1}
Under the above notations, we have
    \begin{align}\label{eq-4.6}
        \liminf\limits_{p\rightarrow+\infty}\Vert M_{p}\Vert_{\infty}\geq 1,
    \end{align}
    and
    \begin{align}\label{eq-4.7}
        \lim\limits_{p\rightarrow +\infty} p\Vert u_{k,p}\Vert_{\infty}=+\infty,\quad k=1,2.
    \end{align}
Furthermore, there exist $C_1, C_2>0$ such that for $p$ large,
    \begin{align}\label{eq-4.9}
        C_1\leq p\int_{\Omega}u_{k,p}^{p+1} dx \leq C_2,
    \end{align}
    \begin{align}\label{eq-4.9-0}
        C_1\leq p\int_{\Omega}u_{k,p}^{p} dx \leq C_2.
    \end{align}
\end{Proposition}
\begin{proof}
Recalling \eqref{fc-lambda1} that $\lambda_1>0$ denotes the first eigenvalue of the operator $-\Delta$ in $H_0^1(\Omega)$, we see from  \eqref{eq-4.8} that for $p>3$,
    \begin{align}
      \lambda_1\int_{\Omega} u_{k,p}^2 dx\leq   &\int_{\Omega} |\nabla u_{k,p}|^2 dx= \mu_{k}\int_{\Omega}u_{k,p}^{p+1}dx+\beta\int_{\Omega}u_{1,p}^{\frac{p+1}{2}}u_{2,p}^{\frac{p+1}{2}} dx\nonumber\\
         \leq& (\mu_{k}+\beta)\Vert M_{p}\Vert_{\infty}^{p-1}\int_{\Omega}u_{k,p}^2 dx, 
    \end{align}
so
    \begin{align}
        \Vert M_{p}\Vert_{\infty}^{p-1}\geq \frac{\lambda_1}{\mu_{k}+\beta},
    \end{align}
    which implies \eqref{eq-4.6}.

  By \eqref{eq-1-2-3}, it follows that
\begin{align}\label{eq-4.2}
    p\int_{\Omega}u_{k,p}^{p+1}dx =\frac{1}{\mu_{k}}p\int_{\Omega} |\nabla u_{k,p}|^2 dx-\frac{\beta}{\mu_{k}} p\int_{\Omega}u_{1,p}^{\frac{p+1}{2}}u_{2,p}^{\frac{p+1}{2}} dx\leq  C.
\end{align}
This, together with H\"older inequality, shows that
\begin{align*}
    p\int_{\Omega}u_{k,p}^{p} dx\leq \left(p\int_{\Omega}u_{k,p}^{p+1} dx \right)^{\frac{p}{p+1}}\left(p|\Omega|\right)^{\frac{1}{p+1}}\leq C.
\end{align*}
On the other hand, by Lemma \ref{lm2.4}, we have
    \begin{align}\label{eq-4.18}
       \left(\int_{\Omega}u_{k,p}^{p+1}dx\right)^{\frac{2}{p+1}}\leq S_p^2 (p+1)\int_{\Omega} |\nabla u_{k,p}|^2 dx,
    \end{align}
    where $S_p^{-2}=8\pi e+o(1)$ as $p\to+\infty$.
    This, together with \eqref{eq-4.2} and H\"older inequality, gives that for $p$ large,
    \begin{align*}
        &(8\pi e+o(1))\left(\int_{\Omega}u_{k,p}^{p+1}dx\right)^{\frac{2}{p+1}}\\
        \leq&(p+1)\int_{\Omega} |\nabla u_{k,p}|^2 dx=\mu_{k}(p+1)\int_{\Omega}u_{k,p}^{p+1}dx+\beta(p+1)\int_{\Omega}u_{1,p}^{\frac{p+1}{2}}u_{2,p}^{\frac{p+1}{2}} dx\nonumber\\
        \leq &(p+1)\left(\mu_{k}\left(\int_{\Omega}u_{k,p}^{p+1}dx\right)^{\frac{1}{2}}+\beta\left(\int_{\Omega}u_{3-k,p}^{p+1}dx\right)^{\frac{1}{2}}\right)\left(\int_{\Omega}u_{k,p}^{p+1}dx\right)^{\frac{1}{2}}\nonumber\\
        \leq & C(p+1)^\frac12\left(\int_{\Omega}u_{k,p}^{p+1}dx\right)^{\frac{1}{2}},\nonumber
    \end{align*}
and so
    \[
 (p+1)\left(\int_{\Omega}u_{k,p}^{p+1}dx\right)^{1-\frac{4}{p+1}}\geq C.
    \]
  Thus 
    \eqref{eq-4.9} holds and
   $$C_1\leq p\int_{\Omega}u_{k,p}^{p+1} dx \leq Cp\int_{\Omega}u_{k,p}^{p} dx,$$
i.e. \eqref{eq-4.9-0} holds. 

Finally, assume by contradiction that \eqref{eq-4.7} does not hold for some $k\in\{1,2\}$, namely up to a subsequence, $p\|u_{k,p}\|_{\infty}\leq C$. Then
$$C_1\leq p\int_{\Omega}u_{k,p}^{p+1} dx\leq C\frac{C^{p+1}}{p^p}\to 0,\quad\text{as }p\to\infty,$$
a contradiction. Thus \eqref{eq-4.7} holds.
This completes the proof.
\end{proof}

Now we turn to prove
the existence of the first concentration point. Recalling $x_{p,1}$ defined in \eqref{eq-4.5}, we assume, up to a subsequence, that 
\begin{equation}\label{fc-x1}x_{p,1}\rightarrow x_1\in\Omega\quad\text{as}\quad p\rightarrow \infty.\end{equation} 
Define
\begin{align}\label{eq-4.20}
    v_{k,p,1}(z):=\frac{p}{M_p(x_{p,1})}\left(u_{k,p}(\varepsilon_{p,1}z+x_{p,1})-M_p(x_{p,1})\right), \quad k=1,2,
\end{align}
where $z\in \widetilde{\Omega}_{p,1}:=\left(\Omega-x_{p,1}\right)/\varepsilon_{p,1}$ and
\begin{align}\label{fc-ep}
    \varepsilon_{p,1}^{-2}:=pM_p(x_{p,1})^{p-1}.
\end{align}
By Proposition \ref{prop-4.1}, we have $\varepsilon_{p,1}\rightarrow 0$. Furthermore, Lemma \ref{lemma2-1} shows 
    \begin{align}\label{fc-partial}
       dist(x_{p,1},\partial\Omega)\geq\delta,\quad\text{and so}\quad \widetilde{\Omega}_{p,1}\to \mathbb R^2.
    \end{align}

\begin{Lemma}\label{lemma-3-5^*}
    Up to a subsequence, one of the following possibilities holds.
    \begin{itemize}
        \item [(1)] $(v_{1,p,1},v_{2,p,1})\rightarrow (V_{1},V_{2})$ uniformly in $C^{2}_{loc}(\mathbb{R}^2)\times C^{2}_{loc}(\mathbb{R}^2)$,
        where 
        \begin{align}\label{eq-3-48}
           & V_{1}(z)+\log\frac{\mu_1}{\mu_2}=V_{2}(z):=-2\log\left({1+\frac{\mu_2}{8}\left(1+\frac{\beta}{\sqrt{\mu_1\mu_2}}\right) |z|^2}\right),\\
           & \int_{\mathbb R^2}e^{V_{k}}dx=\frac{8\pi}{\mu_k(1+\frac{\beta}{\sqrt{\mu_{1}\mu_{2}}})},\quad k=1,2.\nonumber
        \end{align}
        \item [(2)] There is $k\in\{1,2\}$ such that
         $v_{k,p,1}\rightarrow \widetilde V_{k}$ uniformly in $C^{2}_{loc}(\mathbb{R}^2)$, where 
         \begin{align}\label{eq-3-49}
             \widetilde V_{k}(z):=-2\log\left({1+\frac{\mu_{k}}{8}|z|^2}\right),\quad \int_{\mathbb R^2}e^{\widetilde V_{k}}dx=\frac{8\pi}{\mu_k},
         \end{align}
        and $v_{3-k,p,1}\rightarrow -\infty$ uniformly in  any compact subsets of $\mathbb{R}^2$.
    \end{itemize}
\end{Lemma}
\begin{proof}
    By direct computations, we have 
    \begin{align}\label{eq-3-50}
        \begin{cases}
            -\Delta v_{1,p,1}=\mu_1\left(1+\frac{v_{1,p,1}}{p}\right)^p+\beta\left(1+\frac{v_{1,p,1}}{p}\right)^{\frac{p-1}{2}}\left(1+\frac{v_{2,p,1}}{p}\right)^{\frac{p+1}{2}},\\
            -\Delta v_{2,p,1}=\mu_2\left(1+\frac{v_{2,p,1}}{p}\right)^p+\beta\left(1+\frac{v_{2,p,1}}{p}\right)^{\frac{p-1}{2}}\left(1+\frac{v_{1,p,1}}{p}\right)^{\frac{p+1}{2}},\\
            v_{k,p,1}\leq 0,\quad \max\{v_{1,p,1}(0), v_{2,p,1}(0)\}=0,\\
               0\leq 1+\frac{v_{k,p,1}(z)}{p}=\frac{u_{k,p}(\varepsilon_{1,p}z+x_{p,1})}{M_p(x_{p,1})}\leq 1,\quad k=1,2.
        \end{cases}
    \end{align}
    For any $R>1$, we have $B_R(0)\subset\widetilde{\Omega}_{p,1}$ for $p$ large by \eqref{fc-partial}. For $k\in\{1,2\}$, we let 
    \begin{align}
        v_{k,p,1}=\phi_{k,p,1}+\psi_{k,p,1} \quad\text{ in }\,\,B_R(0),
    \end{align}
    where 
    \begin{equation*}
       \begin{cases}
    -\Delta \phi_{k,p,1}=-\Delta v_{k,p,1} \quad\text{in }\; B_R(0),\\
     \phi_{k,p,1} =0\quad\text{on }\;\partial B_R(0),
    \end{cases}
    \,
     \begin{cases}
        -\Delta \psi_{k,p,1}=0 \quad\text{in }\; B_R(0),\\
      \psi_{k,p,1}=v_{k,p,1} \quad\text{on }\;\partial B_R(0).
     \end{cases} 
    \end{equation*} 
  It follows from \eqref{eq-3-50} that $|\Delta \phi_{k,p,1}|\leq \mu_k+\beta$. 
    By the standard elliptic theory, $\phi_{k,p,1}$ is uniformly bounded in $B_R(0)$ for $p$ large. From here and $v_{k,p,1}\leq 0$, we see 
     that $\psi_{k,p,1}$ is uniformly bounded from above in $B_R(0)$. Then by the Harnack inequality, there are two possibilities:
    \begin{itemize}
        \item [(1)] $\psi_{k,p,1}$ is uniformly bounded in $B_{R/2}(0)$, and so is $v_{k,p,1}$.
        \item [(2)] $\psi_{k,p,1}\rightarrow -\infty$ uniformly in $B_{R/2}(0)$, and so is $v_{k,p,1}$.
    \end{itemize}
    Note that once (1) (resp. (2)) holds for some $R>1$, then (1) (resp. (2)) holds for all $R>1$.
    Since $\max\{v_{1,p,1}(0),v_{2,p,1}(0)\}=0$, there is $k\in \{1,2\}$ such that
    $v_{k,p,1}(0)=0$. 
Then the above argument implies that there are two possibilities:
    \begin{itemize}
        \item [(1)] Both $v_{1,p,1}$ and  $v_{2,p,1}$ are locally uniformly bounded  in $\mathbb R^2$ (this case happens if $v_{3-k,p,1}(0)$ is uniformly bounded).
        \item [(2)]  $v_{k,p,1}$ is locally uniformly bounded  in $\mathbb R^2$ and $\psi_{3-k,p,1}\rightarrow -\infty$ uniformly in any compact subset of $\mathbb R^2$ (this case happens if $v_{3-k,p,1}(0)\to-\infty$).
    \end{itemize}
 From here and the standard elliptic theory, up to a subsequence,  we obtain that 
    \begin{itemize}
        \item [(1)] either $(v_{1,p,1},v_{2,p,1})\rightarrow (V_{1},V_{2})$ in $C^{2}_{loc}(\mathbb{R}^2)\times C^{2}_{loc}(\mathbb{R}^2)$,
        where 
        \begin{align}\label{fc-3-26}
        \begin{cases}
        -\Delta V_{1} = \mu_1 e^{V_{1}} +\beta e^{\frac{V_{1}+V_{2}}{2}}  \quad\text{in} \, \mathbb{R}^2,\\
        -\Delta V_{2}= \mu_2 e^{V_{2}} +\beta e^{\frac{V_{1}+V_{2}}{2}}  \quad\text{in} \, \mathbb{R}^2,\\
        V_{1}\leq 0,\quad V_{2}\leq 0,\quad \max\{V_{1}(0), V_{2}(0)\}=0,
        \end{cases}
        \end{align}
        \item [(2)] or
         $v_{k,p,1}\rightarrow \widetilde V_{k}$ in $C^{2}_{loc}(\mathbb{R}^2)$, where 
         \begin{align}\label{fc-3-27}
         \begin{cases}
             -\Delta \widetilde V_{k}=\mu_{k} e^{\widetilde V_{k}} \quad\text{in} \, \mathbb{R}^2,\\
             V_{k,1}\leq 0=V_{k,1}(0),
             \end{cases}
         \end{align}
         and $v_{3-k,p,1}\rightarrow -\infty$ uniformly in any compact subset of $\mathbb{R}^2$.
    \end{itemize}
    For any $R>1$, by the dominated convergence theorem, Proposition \ref{prop-4.1}, \eqref{eq-4.5} and \eqref{fc-ep}, we have that for the first case \eqref{fc-3-26},
    \begin{align}\label{fc-3-29}
        \int_{B_R(0)} e^{V_{k}} dz =&\lim_{p\to+\infty}\int_{B_R(0)}\left(1+\frac{v_{k,p,1}}{p}\right)^pdz\\
       =&  \lim_{p\to+\infty}\int_{B_R(0)} \frac{u_{k,p}(\varepsilon_{p,1}z+x_{p,1})^p}{M_p(x_{p,1})^p} dz\nonumber\\
        = &\lim_{p\to+\infty}\frac{p}{M_p(x_{p,1})}\int_{B_{R\varepsilon_{p,1}}(x_{p,1})} u_{k,p}^p dx\leq C,\nonumber
    \end{align}
    so
    $$\int_{\mathbb R^2} e^{V_{1}} dz<+\infty,\quad \int_{\mathbb R^2} e^{V_{2}} dz<+\infty,$$
    and then the assertion \eqref{eq-3-48} follows directly from Theorem \ref{th-3.1} (Note that we assume $\mu_1\geq \mu_2$, from which and Theorem \ref{th-3.1} we must have $V_{1}(0)\leq V_{2}(0)=0$).
    
    Similarly, for the second case \eqref{fc-3-27}, we have $\int_{\mathbb R^2} e^{\widetilde V_{k}} dz<+\infty$, and then the assertion \eqref{eq-3-49} follows directly from Chen-Li's famous result
 \cite[Theorem 1.1]{Chen-Li-91}.
 
 The proof is complete.
\end{proof}

\begin{remark}\label{remark3-1}
For the first case \eqref{fc-3-26}, it follows from \eqref{fc-3-29} and $\liminf\limits_{p\to+\infty}M_p(x_{p,1})\geq 1$ that
$$\liminf_{p\to+\infty}p\int_{B_{R\varepsilon_{p,1}}(x_{p,1})} u_{k,p}^p dx=\liminf_{p\to+\infty}M_p(x_{p,1}) \int_{B_R(0)} e^{V_{k}} dz\geq  \int_{B_R(0)} e^{V_{k}} dz,$$
and so for any $r>0$,
$$\liminf_{p\to+\infty}p\int_{B_{r}(x_{1})\cap\Omega} u_{k,p}^p dx\geq\liminf_{p\to+\infty}M_p(x_{p,1}) \int_{\mathbb R^2} e^{V_{k}} dz\geq  \int_{\mathbb R^2} e^{V_{k}} dz.$$
Similarly, we can obtain
$$\liminf_{p\to+\infty}p\int_{B_{R\varepsilon_{p,1}}(x_{p,1})} u_{k,p}^{p+1} dx=\liminf_{p\to+\infty}M_p(x_{p,1})^2 \int_{B_R(0)} e^{V_{k}} dz\geq  \int_{B_R(0)} e^{V_{k}} dz,$$
and so for any $r>0$,
$$\liminf_{p\to+\infty}p\int_{B_{r}(x_{1})\cap\Omega} u_{k,p}^{p+1} dx\geq\liminf_{p\to+\infty}M_p(x_{p,1})^2 \int_{\mathbb R^2} e^{V_{k}} dz\geq  \int_{\mathbb R^2} e^{V_{k}} dz.$$
For the second case \eqref{fc-3-27}, the above estimates also hold by replacing $V_k$ with $\widetilde V_k$.
\end{remark}

Recalling \eqref{eq-4.4} that $M_p=\max\{u_{1,p},u_{2,p}\}$, we define the blowup set of $pu_{1,p}$ and $pu_{2,p}$ as
\begin{align*}
    \mathcal{S}:=\{x\in\overline{\Omega}:\text{ there exist subsequences }\{p_n\}\ \text{and\ }\{x_n\}\subset\Omega \nonumber\\
    \text{  such that } x_n  \rightarrow x \text{ and } p_n M_{p_n}(x_n) \rightarrow +\infty \text{ as } p_n \rightarrow +\infty\}.
\end{align*} 
Then it follows from \eqref{fc-x1} that $x_1\in \mathcal S$. To show that $\mathcal S$ is a finite set, 
we assume that there exist $n\geq 1$ families
of points $(x_{p,j})_{p\geq p_0}\in\Omega$ (We always take $x_{p,1}$ to be given by \eqref{eq-4.5}) such that as $p\rightarrow+\infty$, \[\varepsilon_{p,j}^{-2}:=p M_p(x_{p,j})^{p-1}\rightarrow+\infty,\quad j=1,2,\ldots,n.\]
It follows that for $j=1,2,\ldots,n$,
\[\liminf_{p\to\infty}M_p(x_{p,j})\geq1,\] and up to a subsequence,
\[\{\lim\limits_{p\rt\iy} x_{p,j} :j=1,2,\ldots,n\} \subset \s  \subset\overline{\om}.\]
Let us define 
    \begin{align}\label{fc-3-vkp}
    v_{k,p,j}(z):=\frac{p}{M_p(x_{p,j})}\left(u_{k,p}(\varepsilon_{p,j}z+x_{p,j})-M_p(x_{p,j})\right),\quad k=1,2,
    \end{align} 
for $z\in \widetilde{\Omega}_{p,j}:=\left(\Omega-x_{p,j}\right)/\varepsilon_{p,j}$, and introduce the following properties:
\begin{itemize}
    \item [($\mathcal{P}_1^n$)] For any $j,i\in\{1,2,\ldots,n\},$ $j\neq i,$ $\lim\limits_{p\rt\iy}\frac{|x_{p,j}-x_{p,i}|}{\e_{p,j}}=+\infty$.
    \item [($\mathcal{P}_2^n$)]  For any $j\in\{1,2,\ldots,n\}$,
    after taking a subsequence, one of the following cases happens:
    \begin{itemize}
    \item [Type $\mathcal{A}$ ] 
         $(v_{1,p,j},v_{2,p,j})\rightarrow (V_{1},V_{2})$ in $C^2_{loc}(\mathbb{R}^2)\times C^2_{loc}(\mathbb{R}^2)$,
    where $(V_1, V_2)$ is given by \eqref{eq-3-48}, i.e. \begin{align}\label{fc-3-31}
 & V_{1}(z)+\log\frac{\mu_1}{\mu_2}=V_{2}(z)=-2\log\left({1+\frac{\mu_2}{8}\left(1+\frac{\beta}{\sqrt{\mu_1\mu_2}}\right) |z|^2}\right),\\
           & \int_{\mathbb R^2}e^{V_{k}}dx=\frac{8\pi}{\mu_k(1+\frac{\beta}{\sqrt{\mu_{1}\mu_{2}}})},\quad k=1,2.\nonumber
         \end{align}
    \item [Type $\mathcal{B}$]  
    $v_{1,p,j}\rightarrow \widetilde V_{1}$ in $C^2_{loc}(\mathbb{R}^2)$, where $\widetilde V_{1}$ is given by \eqref{eq-3-49}, i.e.
         \begin{align}
             \widetilde V_{1}(z)=-2\log\left({1+\frac{\mu_1}{8}|z|^2}\right),\quad \int_{\mathbb R^2}e^{\widetilde V_{1}}dx=\frac{8\pi}{\mu_1},
         \end{align}
        and $v_{2,p,j}\rightarrow -\infty$ uniformly in  any compact subsets of $\mathbb{R}^2$.
    \item [Type $\mathcal{C}$]
     $v_{2,p,j}\rightarrow \widetilde V_{2}$ in $C^2_{loc}(\mathbb{R}^2)$, where 
         \begin{align}\label{fc-3-33}
             \widetilde V_{2}(z)=-2\log\left({1+\frac{\mu_2}{8}|z|^2}\right), \quad \int_{\mathbb R^2}e^{\widetilde V_{2}}dx=\frac{8\pi}{\mu_2},
         \end{align}
        and $v_{1,p,j}\rightarrow -\infty$ uniformly in  any compact subsets of $\mathbb{R}^2$.
    \end{itemize}
    \item [($\mathcal{P}_3^n$)]  There exists $C>0$ such that $p R_{p,n}(x)^2M_p(x)^{p-1}\leq C$ for all $p\geq p_0$ and $x\in\om$, where
    \begin{align}\label{eq-3-61}
        R_{p,n}(x):=\min\limits_{1\leq j\leq n}|x-x_{p,j}|.
    \end{align}
\end{itemize}
Lemma \ref{lemma-3-5^*} shows that $(\mathcal{P}_1^1)$-$(\mathcal{P}_2^1)$ hold. Note also that if $(\mathcal{P}_3^n)$ holds for some $n\geq 1$, then we can not find the $(n+1)$-th family of points $(x_{p,n+1})_{p\geq p_0}\in \om$ such that $\varepsilon_{p,n+1}^{-2}:=pM_p(x_{p,n+1})^{p-1}\to\infty$ and $(\mathcal{P}_1^{n+1})$ holds, i.e. $(\mathcal{P}_1^{n+1})$-$(\mathcal{P}_2^{n+1})$ can not hold.
The following result shows the finiteness of $\mathcal{S}$.
\begin{Proposition}\label{prop-3-6}
There exist $n_0\geq 1$ and $n_0$ families of points $(x_{p,j})_{p\geq p_0}$ in $\om,$ $j=1,\ldots,n_0$ such that up to a subsequence, the following statements hold.
\begin{itemize}
\item[(1)] $(\p_1^{n_0}),\ (\p_2^{n_0})$ and $(\p_3^{n_0})$ hold. Consequently, we can not find the $(n_0+1)$-th family of points $(x_{p,n_0+1})_{p\geq p_0}\in \om$ such that $(\p_1^{n_0+1})$-$ (\p_2^{n_0+1})$ hold.
\item[(2)] \begin{equation}\label{fc-s}\s=\{\lim\limits_{p\rt\iy} x_{p,j} :j=1,2,\ldots,n_0\},\end{equation}
so $\s$ is finite.
Furthermore, given any compact subset $\mathscr{K}\Subset\overline{\Omega}\setminus\s$, there exists a constant $C(\mathscr{K})>0$ independent of $p$ such that
\begin{align}\label{31-1}
  p \Vert u_{k,p}\Vert_{L^\infty(\mathscr{K})}\leq C(\mathscr{K}),\quad k=1,2,
\end{align}
\begin{align}\label{31}
  p \Vert\nabla u_{k,p}\Vert_{L^\infty(\mathscr{K})}\leq C(\mathscr{K}),\quad k=1,2.
\end{align}
\end{itemize}
\end{Proposition}
\begin{proof}
    In this proof, we adopt some ideas from  \cite[Proposition 2.2]{Franc-Isabella-Filomena-15} and \cite[Theorem 1.1]{L-W-Z-13}. We divide the proof into several steps.

    \textbf{Step 1.} We prove that for some $n\geq 1$, if $(\p_1^n)$-$(\p_2^n)$ hold, then either $(\p_1^{n+1})$-$(\p_2^{n+1})$ hold or $(\p_3^{n+1})$ holds. 
    
    Assume that $(\p_1^n)$-$(\p_2^n)$ hold but $ (\p_3^{n+1})$ fails, then up to a subsequence, there exists 
    $x_{p,n+1}\in\Omega$ such that
    \begin{align}\label{eq-3-64}
        p R_{p,n}(x_{p,n+1})^2M_p(x_{p,n+1})^{p-1}:=\max_{x\in\Omega} p R_{p,n}(x)^2M_p(x)^{p-1}\to+\infty.
    \end{align}
    Define
    \begin{align}\label{fc-vpn+1}
        \varepsilon_{p,n+1}^{-2}:=p M_p(x_{p,n+1})^{p-1}\to+\infty,
    \end{align}
and
    \begin{align*}
        v_{k,p,n+1}(z):=\frac{p}{M_p(x_{p,n+1})}\left(u_{k,p}(\varepsilon_{p,n+1}z+x_{p,n+1})-M_p(x_{p,n+1})\right),\; k=1,2,
    \end{align*}
    for $z\in \widetilde{\Omega}_{p,n+1}:=\left(\Omega-x_{p,n+1}\right)/\varepsilon_{p,n+1}$.  Similarly as \eqref{eq-3-50}, we have
    \begin{align}\label{fc-3-41}
        \begin{cases}
            -\Delta v_{1,p,n+1}=\mu_1\left(1+\frac{v_{1,p,n+1}}{p}\right)^p+\beta\left(1+\frac{v_{1,p,n+1}}{p}\right)^{\frac{p-1}{2}}\left(1+\frac{v_{2,p,n+1}}{p}\right)^{\frac{p+1}{2}},\\
            -\Delta v_{2,p,n+1}=\mu_1\left(1+\frac{v_{2,p,n+1}}{p}\right)^p+\beta\left(1+\frac{v_{2,p,n+1}}{p}\right)^{\frac{p-1}{2}}\left(1+\frac{v_{1,p,n+1}}{p}\right)^{\frac{p+1}{2}},\\
            1+\frac{v_{k,p,n+1}(z)}{p}=\frac{u_{k,p}(\varepsilon_{p,n+1}z+x_{p,n+1})}{M_p(x_{p,n+1})}\geq 0, \; \max\{v_{1,p,n+1}(0), v_{2,p,n+1}(0)\}=0.
        \end{cases}
    \end{align}
 
 {\bf Step 1-1.} We prove that  $(\p_1^{n+1})$ holds.   
 
 First, we prove that for any $i\in\{1,\ldots,n\}$,
    \begin{align}\label{eq-3-67}
        \lim\limits_{p\rt\iy}\frac{|x_{p,i}-x_{p,n+1}|}{\e_{p,i}}=+\infty.
    \end{align}
    Assume by contradiction that up to a subsequence, there exists $y_0\in\mathbb{R}^2$ such that 
    \begin{align*}
        \frac{x_{p,n+1}-x_{p,i}}{\e_{p,i}}\to y_0,\quad\text{for some }\;i\in\{1,\ldots,n\}.
    \end{align*}
    We may also assume $M_p(x_{p,n+1})=u_{k_0,p}(x_{p,n+1})$ for some $k_0\in \{1,2\}$.
 Then by $(\p_2^n)$ and
 $$1+\frac{v_{k,p,i}(z)}{p}=\frac{u_{k,p}(\varepsilon_{p,i}z+x_{p,i})}{M_p(x_{p,i})}\geq 0,$$
we have
    \begin{align*}
        &p|x_{p,i}-x_{p,n+1}|^2 M_p(x_{p,n+1})^{p-1}\nonumber\\
        =&p|x_{p,i}-x_{p,n+1}|^2 u_{k_0,p}(x_{p,n+1})^{p-1}\nonumber\\
        =&\left(\frac{|x_{p,n+1}-x_{p,i}|}{\e_{p,i}}\right)^2 \left(1+\frac{v_{k_0,p,i}(\frac{x_{p,n+1}-x_{p,i}}{\e_{p,i}})}{p}\right)^{p-1}\\
        \leq&|y_0|^2e^{\max\{V_{1}(y_0), V_{2}(y_0), \widetilde{V}_1(y_0), \widetilde{V}_2(y_0)\}}+o(1)\leq C\nonumber,
    \end{align*}
  for $p$ large,  which is a contradiction with \eqref{eq-3-64}. This proves \eqref{eq-3-67}.
Furthermore, it follows  from \eqref{eq-3-64} that 
    \begin{align}\label{eq-3-70}
        \frac{|x_{p,i}-x_{p,n+1}|}{\e_{p,n+1}}\geq \left(p R_{p,n}(x_{p,n+1})^2M_p(x_{p,n+1})^{p-1}\right)^{\frac{1}{2}}\rightarrow+\infty, \quad\forall i\leq n.
    \end{align}
    Thus, $(\p_1^{n+1})$ follows from \eqref{eq-3-67}, \eqref{eq-3-70} and $(\p_1^n)$.

 {\bf Step 1-2.}  We prove the convergence of $v_{k,p,n+1}$ such that $(\p_2^{n+1})$ holds.
    
Notice that comparing to \eqref{eq-4.5} and Lemma \ref{lemma-3-5^*}, the different thing here is that $x_{p,n+1}$ is no longer a local maximum point of $M_{p}$, so we need to control $M_p$ by $M_p(x_{p,n+1})$ near $x_{p+1}$.     
Fix any $\eta\in (0,1)$ and any large $R>1$.
    By \eqref{eq-3-70}, we have that for $p$ large and for any $z\in B_R(0)\cap\widetilde{\Omega}_{p,n+1}$,
    \begin{align*}
        \frac{|x_{p,i}-x_{p,n+1}|}{\varepsilon_{p,n+1}}\geq \frac{R}{\eta} \geq \frac{|z|}{\eta},\quad\forall i\leq n,
    \end{align*}
    \begin{align*}
        |x_{p,i}-x_{p,n+1}-\varepsilon_{p,n+1}z|
        \geq  (1-\eta)|x_{p,i}-x_{p,n+1}|,\;\forall i\leq n,
    \end{align*}
and so
    \begin{align}
        0<\frac{R_{p,n}(x_{p,n+1})}{R_{p,n}(\varepsilon_{p,n+1}z+x_{p,n+1})}\leq \frac{1}{1-\eta}.
    \end{align}
    From here and \eqref{eq-3-64}, we obtain that for $p$ large and for any $z\in B_R(0)\cap\widetilde{\Omega}_{p,n+1}$,
    \begin{align}\label{eq-3-74}
       \left(\frac{M_p(\varepsilon_{p,n+1}z+x_{p,n+1})}{M_p(x_{p,n+1})}\right)^{p-1}\leq &\frac{R_{p,n}(x_{p,n+1})^2}{R_{p,n}(\varepsilon_{p,n+1}z+x_{p,n+1})^2}
       \leq \frac{1}{(1-\eta)^2}.
    \end{align}
   This, together with \eqref{fc-3-41}, implies that $|\Delta v_{1,p,n+1}|$ and $|\Delta v_{2,p,n+1}|$ are both uniformly bounded in $B_R(0)\cap\widetilde{\Omega}_{p,n+1}$ for $p$ large. Recall from \eqref{fc-vpn+1} that $\varepsilon_{p,n+1}\rightarrow 0$ as $p\rightarrow\infty$ and
    \begin{align}\label{fc-3-43}
        \liminf_{p\rightarrow\infty}M_p(x_{p,n+1})\geq 1.
    \end{align}
    Since $t^{p}-1-p(t-1)\geq 0$ for all $t\geq 0$, we see from
    \eqref{eq-3-74} that for $p$ large and for any $z\in B_R(0)\cap\widetilde{\Omega}_{p,n+1}$,
    \begin{align}
        v_{k,p,n+1}(z)\leq & p\left(\frac{M_p(\varepsilon_{p,n+1}z+x_{p,n+1})}{M_p(x_{p,n+1})}-1\right)\\
        \leq &\left(\frac{M_p(\varepsilon_{p,n+1}z+x_{p,n+1})}{M_p(x_{p,n+1})}\right)^{p}-1\leq C.\nonumber
    \end{align}
    On the other hand, we will prove
    \begin{equation}\label{fc-xpn}
    \frac{dist(x_{p,n+1},\partial\Omega)}{\varepsilon_{p,n+1}}\to+\infty,\quad\text{and so }\;\widetilde{\Omega}_{p,n+1}\to\mathbb R^2.
    \end{equation}
    Once \eqref{fc-xpn} is proved, thanks to \eqref{fc-3-41} and \eqref{eq-3-74}-\eqref{fc-xpn}, and since $R>1$ is arbitrary, we can repeat the proof of
    Lemma \ref{lemma-3-5^*} to obtain the convergence of $v_{k,p,n+1}$. For example, if both $v_{1,p,n+1}(0)$ and $v_{2,p,n+1}(0)$ are uniformly bounded, then up to a subsequence, $(v_{1,p,n+1},v_{2,p,n+1})\rightarrow (V_{1,n+1},V_{2,n+1})$ in $C^2_{loc}(\mathbb{R}^2)\times C^2_{loc}(\mathbb{R}^2)$, where $(V_{1,n+1},V_{2,n+1})$ is a solution of the Liouville-type system \eqref{eq-C-Z}.
    Since \eqref{eq-3-74} gives that for each fixed $z$,
    \[\left(1+\frac{v_{k,p,n+1}(z)}{p}\right)^{p-1}=\left(\frac{u_{k,p}(\varepsilon_{p,n+1}z+x_{p,n+1})}{M_p(x_{p,n+1})}\right)^{p-1}\leq \frac{1}{(1-\eta)^2},\]
    letting $p\to\infty$, we obtain $e^{V_{k,n+1}(z)}\leq \frac{1}{(1-\eta)^2}$. Since $\eta\in (0,1)$ is arbitrary, we have  $e^{V_{k,n+1}(z)}\leq 1$ and so $V_{k,n+1}(z)\leq 0$. Thus
   \[ V_{1,n+1}\leq 0,\quad V_{2,n+1}\leq 0,\quad \max\{V_{1,n+1}(0), V_{2,n+1}(0)\}=0.\]
   Then the same proof of   Lemma \ref{lemma-3-5^*} implies that $(V_{1,n+1}, V_{2,n+1})=(V_1, V_2)$, where $(V_1, V_2)$ is given by \eqref{fc-3-31}. The other case that $v_{k,p,n+1}(0)$ is uniformly bounded and $v_{3-k,p,n+1}(0)\to-\infty$ for some $k\in\{1,2\}$ can be discussed similarly. This proves that $(\p_2^{n+1})$ holds.
   
   Now we turn to prove \eqref{fc-xpn}, and in the proof we omit the subscript $n+1$ in $v_{k,p,n+1}, \varepsilon_{p,n+1}, x_{p,n+1}, \widetilde{\Omega}_{p,n+1}$ just for convenience.  Assume by contradiction that up to a subsequence, $dist(x_{p},\partial\Omega)=O(\varepsilon_{p})$, then up to a rotation, we may assume $\widetilde{\Omega}_{p}\rightarrow (-\infty,t_0)\times\mathbb{R}$ for some $0\leq t_0<\infty $. Recalling \eqref{fc-3-41}, without loss of generality, we may assume
   $$v_{1,p}(0)=0.$$
    Fix any $R_0\geq t_0+1$. Then for any $z\in B_{R_0}(0)\cap\widetilde{\Omega}_{p}$, by the Green’s representation formula and \eqref{eq-2.1}, we have
    \begin{align}\label{eq-4-23}
       |\nabla v_{1,p}(z)|=&\frac{p\varepsilon_{p}}{M_p(x_{p})}|\nabla u_{1,p}(x_{p}+\varepsilon_{p}z)|\nonumber\\
       =&\frac{p\varepsilon_{p}}{M_p(x_{p})}\left|\int_{\Omega}\nabla_x G(x_{p}+\varepsilon_{p}z,y)\left(\mu_1 u_{1,p}^{p}+\beta u_{1,p}^{\frac{p-1}{2}}u_{2,p}^{\frac{p+1}{2}}\right)dy\right|\\
       \leq & C \frac{  p\varepsilon_{p}}{M_p(x_{p})}\int_{\Omega}\frac{1}{|x_{p}+\varepsilon_{p}z-y|}\left(\mu_1 u_{1,p}^{p}+\beta u_{1,p}^{\frac{p-1}{2}}u_{2,p}^{\frac{p+1}{2}}\right)dy.\nonumber
    \end{align}
    Note that for $y\in \Omega\setminus B_{2R_0\varepsilon_{p}}(x_{p})$, 
    $$|x_{p}+\varepsilon_{p}z-y|\geq |x_{p}-y|-\varepsilon_{p}|z|\geq \varepsilon_{p}R_0,$$ so it follows from \eqref{eq-4.9-0} and \eqref{fc-3-43} that 
    \begin{align}\label{eq-4.24}
        \frac{  p\varepsilon_{p}}{M_p(x_{p})}&\int_{\Omega\setminus B_{2R_0\varepsilon_{p}}(x_{p})}\frac{1}{|x_{p}+\varepsilon_{p}z-y|}\left(\mu_1 u_{1,p}^{p}+\beta u_{1,p}^{\frac{p-1}{2}}u_{2,p}^{\frac{p+1}{2}}\right)dy\\
        \leq \frac{p}{M_p(x_{p})R_0}&\int_{\Omega}\mu_1 u_{1,p}^{p}+\beta u_{1,p}^{\frac{p-1}{2}}u_{2,p}^{\frac{p+1}{2}} dy\leq C.\nonumber
    \end{align}
    For $y\in B_{2R_0\varepsilon_{p}}(x_{p})$, using \eqref{eq-3-74} with $R=2R_0$, we have 
    $$\max\{u_{1,p}(y), u_{2,p}(y)\}\leq (1-\eta)^{-\frac{2}{p-1}}M_p(x_p),$$
    This together with \eqref{fc-vpn+1} implies that (write $y=x_{p}+\varepsilon_{p}\tau$)
    \begin{align}\label{eq-4.25}
        &\frac{ p\varepsilon_{p}}{M_p(x_{p})}\int_{B_{2R_0\varepsilon_{p}}(x_{p})}\frac{1}{|x_{p}+\varepsilon_{p}z-y|}\left(\mu_1 u_{1,p}^{p}+\beta u_{1,p}^{\frac{p-1}{2}}u_{2,p}^{\frac{p+1}{2}}\right)dy\\
        \leq &C p\varepsilon_{p} M_p(x_{p})^{p-1} \int_{B_{2R_0\varepsilon_{p}}(x_{p})}\frac{1}{|x_{p}+\varepsilon_{p}z-y|} dy\nonumber\\
        = & C\int_{B_{2R_0}(0)}\frac{1}{|\tau-z|}d\tau\leq C.\nonumber
    \end{align}
 Inserting \eqref{eq-4.24} and \eqref{eq-4.25} into \eqref{eq-4-23}, we get $|\nabla v_{1,p}(z)|\leq C$ for any $z\in B_{R_0}(0)\cap\widetilde{\Omega}_{p}$, so
 $$|v_{1,p}(z)|=|v_{1,p}(z)-v_{1,p}(0)|\leq C,\quad\forall z\in\overline{B_{R_0}(0)\cap\widetilde{\Omega}_{p}}.$$
 From here and
 $R_0\geq t_0+1$, for $p$ large
    there exists $z_p\in\partial\widetilde{\Omega}_{p}\cap B_{R_0}(0)$ such that $p=|v_{1,p}(z_p)|\leq C$, a contradiction with $p\to+\infty$. This proves \eqref{fc-xpn}.

    \textbf{Step 2.} We prove there exists $n_0\geq 1$ such that $(\p_1^{n_0}),\ (\p_2^{n_0})$ and $(\p_3^{n_0})$ hold. 
    
    Assume by contradiction that such $n_0$ does not exist, then it follows from Step 1 that 
    $(\p_1^{n})$ and $(\p_2^{n})$ hold for any $n\geq 1$. Recalling \eqref{fc-3-31}-\eqref{fc-3-33} and letting $b_0=\frac{4\pi}{\mu_1(1+\frac{\beta}{\sqrt{\mu_{1}\mu_{2}}})}$, then there is $R_1>1$ such that
    \[\min\Big\{\int_{B_{R_1}(0)}e^{V_{k}}dz,\; \int_{B_{R_1}(0)}e^{\widetilde V_{k}}dz\Big\}\geq b_0,\quad k=1,2.\]
    Recall \eqref{eq-4.9-0} that
    $$p\int_{\Omega}u_{k,p}^{p} dx \leq C_2,\quad k=1,2.$$
    We fix $n>\frac{2C_2}{b_0}$.
    Then by  $(\p_1^{n})$, for $p$ large we have that for any $i,j\in\{1,\ldots,n\}$ with $i\neq j$,
    \begin{align}
        B_{\varepsilon_{p,i} R_1}(x_{p,i})\cap B_{\varepsilon_{p,j} R_1}(x_{p,j})=\emptyset.
    \end{align}
    Clearly those estimates in Remark \ref{remark3-1} also hold for any $j$, so 
     we obtain
    \begin{align*}
    2C_2&\geq \liminf_{p\to+\infty}p\int_{\Omega}u_{1,p}^{p} +u_{2,p}^{p} dx\\
    &\geq \liminf_{p\to+\infty}\sum_{j=1}^np\int_{B_{\varepsilon_{p,j} R_1}(x_{p,j})}u_{1,p}^{p} +u_{2,p}^{p} dx\\
    &\geq n \min\Big\{\int_{B_{R_1}(0)}e^{V_{k}}dz, \int_{B_{R_1}(0)}e^{\widetilde V_{k}}dz\Big\}\geq nb_0>2C_2,
    \end{align*}
    a contradiction. 
    
    Therefore, there exists $n_0\geq 1$ such that $(\p_1^{n_0}),\ (\p_2^{n_0})$ and $(\p_3^{n_0})$ hold. 
 Consequently, \eqref{fc-s} follows directly from the definition of $\s$ and $(\p_3^{n_0})$, and \eqref{31-1} follows directly from \eqref{fc-s} and $(\p_3^{n_0})$.   
   
    \textbf{Step 3.} To prove \eqref{31}, we only need to prove that for any $x\in\Omega\setminus\{x_{p,1},\cdots, x_{p,{n_0}}\}$ and $k=1,2$,
    \begin{align}\label{eq-3-80}
    p\left|\nabla u_{k,p}(x)\right|=O\left(\sum\limits_{j=1}^{n_0}\frac{1}{|x-x_{p,j}|}\right).
    \end{align}
 
  Let us take $k=1$ for example. Define
    \begin{align*}
        \Omega_{p,j}:=\Big\{x\in\Omega\;:\;R_{p,n_0}(x)=\min_{1\leq i\leq n_0}|x_{p,i}-x|=|x_{p,j}-x|\Big\},\;\forall j,
    \end{align*} 
    so $\cup_j\Omega_{p,j}=\Omega$. Then by the Green’s representation formula and \eqref{eq-2.1}, 
    \begin{align}\label{eq-3-81}
        \left|p\nabla u_{1,p}(x)\right|=&\left|p\int_{\Omega}\nabla_x G(x,y) \left(\mu_1 u_{1,p}^p+\beta u_{1,p}^{\frac{p-1}{2}}u_{2,p}^{\frac{p+1}{2}}\right) dy \right|\\
        \leq & C p\int_{\Omega}\frac{1}{|x-y|} \left(\mu_1 u_{1,p}^p+\beta u_{1,p}^{\frac{p-1}{2}}u_{2,p}^{\frac{p+1}{2}}\right)  dy\nonumber\\
        \leq & C  p \sum\limits_{j=1}^{n_0}\left(\int_{\Omega_{p,j}}\frac{1}{|x-y|} u_{1,p}^p dy +\int_{\Omega_{p,j}}\frac{1}{|x-y|} u_{2,p}^p dy \right).\nonumber
    \end{align}
    For $y\in\Omega_{p,j}^1:=\Omega_{p,j}\cap B_{{|x_{p,j}-x|}/{2}}(x_{p,j})$, we have $$|x-y|\geq |x-x_{p,j}|-|y-x_{p,j}|\geq {|x_{p,j}-x|}/{2},$$ which together with  \eqref{eq-4.9-0} implies
    \begin{align}\label{eq-3-84}
        p\int_{\Omega_{p,j}^1}\frac{1}{|x-y|} u_{k,p}^p dy\leq \frac{C}{|x_{p,j}-x|} p\int_{\Omega_{p,j}} u_{k,p}^p dy\leq \frac{C}{|x_{p,j}-x|}.
    \end{align}
    For $y\in\Omega_{p,j}^2:=\Omega_{p,j}\setminus B_{{2|x_{p,j}-x|}}(x_{p,j})$,  we have $$|x-y|\geq |y-x_{p,j}|-|x-x_{p,j}|\geq {|x_{p,j}-x|},$$ which also implies
    \begin{align}\label{eq-3-85}
        p\int_{\Omega_{p,j}^2}\frac{1}{|x-y|} u_{k,p}^p dy\leq \frac{C}{|x_{p,j}-x|} p\int_{\Omega_{p,j}} u_{k,p}^p dy\leq \frac{C}{|x_{p,j}-x|}.
    \end{align}
    For $y\in\Omega_{p,j}^3:=\Omega_{p,j}\cap( B_{{2|x_{p,j}-x|}}(x_{p,j})\setminus B_{{|x_{p,j}-x|}/{2}}(x_{p,j}) )$, by $R_{p,n_0}(y)=|y-x_{p,j}|$ and $(\p_3^{n_0})$, we have
    \begin{align*}
       p|y-x_{p,j}|^2 u_{k,p}(y)^{p-1}\leq p R_{p,n}(y)^2M_p(y)^{p-1}\leq C,
    \end{align*}
    which together with $\|u_{k,p}\|_{\infty}\leq C$ implies
    \begin{align}\label{eq-3-87}
         &p\int_{\Omega_{p,j}^3}\frac{1}{|x-y|} u_{k,p}^pdy
         \leq Cp\int_{\Omega_{p,j}^3}\frac{1}{|x-y|} u_{k,p}^{p-1}dy\nonumber\\
         \leq & C\int_{B_{{2|x_{p,j}-x|}}(x_{p,j})\setminus B_{{|x_{p,j}-x|}/{2}}(x_{p,j})}\frac{dy}{|x-y||y-x_{p,j}|^2} \\
         \leq &\frac{C}{|x_{p,j}-x|^2}\int_{B_{{2|x_{p,j}-x|}}(0)\setminus B_{{|x_{p,j}-x|}/{2}}(0)}\frac{dy}{|y-(x-x_{p,j})|}\nonumber\\
         \leq &\frac{C}{|x_{p,j}-x|^2}\int_{B_{{3|x_{p,j}-x|}}(0)}\frac{dy}{|y|}\leq \frac{C}{|x_{p,j}-x|}.\nonumber
     \end{align}
     Inserting \eqref{eq-3-84}-\eqref{eq-3-87} into \eqref{eq-3-81}, we obtain 
 \eqref{eq-3-80} with $k=1$.

     The proof is complete.
\end{proof}

Since $\mathcal{S}$ is finite, we can write $$\mathcal{S}=\{x_1,\ldots,x_N\}, \quad\text{for some}\; N\leq n_0.$$ 
Take $r_0>0$ small such that for any $x_i\neq x_j\in\mathcal{S}$,
\begin{align}\label{fc-3-57}
    B_{2r_0}(x_i)\cap B_{2r_0}(x_j)=\emptyset.
\end{align}

\begin{Lemma}\label{lemma-3-7}
Up to a subsequence, we have that for $k=1,2,$
\begin{align}
    p u_{k,p}\rightarrow \sum\limits_{i=1}^N \gamma_{k,i} G(\cdot, x_i)\quad\text{ in }\,\,C^{2}_{loc}(\overline{\Omega}\setminus\mathcal{S}),
\end{align}
where
\begin{align*}
    \gamma_{k,i}:=\lim\limits_{\rho\rightarrow 0}\lim\limits_{p \rightarrow \infty} p\int_{B_{\rho}(x_i)\cap\Omega} \left(\mu_k u_{k,p}^p +\beta u_{k,p}^{\frac{p-1}{2}}u_{3-k,p}^{\frac{p+1}{2}}\right) dx.
\end{align*}
\end{Lemma}
\begin{proof} The proof is standard
and we sketch it here for later usage. 
Let $\mathscr{K}\Subset\overline{\om}\backslash\s$ be any compact subset and take $r\in (0,r_0)$ small such that
\begin{align}\label{new3.7}
    \mathscr{K}\subset\overline{\Omega}\Big\backslash\left(\bigcup\limits_{i=1}^N B_{2r}(x_i)\right).
\end{align} For $x\in\mathscr{K}$ and $m\in \{0,1\}$,
by the Green's representation formula, we have that for $\rho\in (0,r)$,
\begin{align}\label{3-22}
    p\nabla^m u_{1,p}(x)&=p \int_{\om}\nabla_x^m G(x,y)\left(\mu_1 u_{1,p}^p +\beta u_{1,p}^{\frac{p-1}{2}}u_{2,p}^{\frac{p+1}{2}}\right)dy\nonumber\\
    &=\sum_{i=1}^{N} \nabla_x^m G(x,x_i) p \int_{\Omega\cap B_\rho(x_i)}\left(\mu_1 u_{1,p}^p +\beta u_{1,p}^{\frac{p-1}{2}}u_{2,p}^{\frac{p+1}{2}}\right) dy\\
    &+\sum_{i=1}^{N} p \int_{\Omega\cap B_\rho(x_i)}\nabla_x^m (G(x,y)- G(x,x_i))\left(\mu_1 u_{1,p}^p +\beta u_{1,p}^{\frac{p-1}{2}}u_{2,p}^{\frac{p+1}{2}}\right) dy\nonumber\\
    &+p \int_{\om\backslash \bigcup_{i=1}^{N} B_\rho(x_i)} \nabla_x^m G(x,y)\left(\mu_1 u_{1,p}^p +\beta u_{1,p}^{\frac{p-1}{2}}u_{2,p}^{\frac{p+1}{2}}\right) dy.\nonumber
\end{align}
Since \eqref{31-1} implies
\begin{align}\label{new3.8}
    \Vert u_{k,p}\Vert_{L^{\iy}(\om \backslash \bigcup_{i=1}^{N} B_{\rho}(x_i))}\leq \frac{C_\rho}{p},\quad k=1,2,
\end{align}
and \eqref{eq-2.1} implies 
\[
    \int_{\om\backslash \bigcup_{i=1}^{N} B_\rho(x_i)} |\nabla_x^m G(x,y)|dy\leq \int_{\om} |\nabla_x^m G(x,y)|dy \leq C,
\]
we obtain
\begin{align}\label{995}
   \left|p \int_{\om\backslash \bigcup_{i=1}^{N} B_\rho(x_i)} \nabla_x^m G(x,y)\left(\mu_1 u_{1,p}^p +\beta u_{1,p}^{\frac{p-1}{2}}u_{2,p}^{\frac{p+1}{2}}\right) dy\right|\leq C\frac{C_\rho^{p}}{p^{p-1}}.
\end{align}
Furthermore, it follows from the fact (note \eqref{new3.7})
\begin{align}\label{new3.11}
    |\nabla_y\nabla_x^mG(x,y)|\leq C,\quad\forall x\in\mathscr{K},\; y\in \cup_{i=1}^N B_{r}(x_i),
\end{align}
and \eqref{eq-4.9-0} that
\begin{align}
    &p\left|\int_{\Omega\cap B_\rho(x_i)}\nabla_x^m (G(x,y)- G(x,x_i))\left(\mu_1 u_{1,p}^p +\beta u_{1,p}^{\frac{p-1}{2}}u_{2,p}^{\frac{p+1}{2}}\right)dy\right|\label{new3.12}\\
    \leq& p\int_{\Omega\cap B_\rho(x_i)}\rho|\nabla_y\nabla_x^m G(x,\lambda_y y+(1-\lambda_y)x_i)| \left(\mu_1 u_{1,p}^p +\beta u_{1,p}^{\frac{p-1}{2}}u_{2,p}^{\frac{p+1}{2}}\right) dy\leq C\rho,\nonumber
\end{align}
 where $\lambda_y\in (0,1)$ depends on $y$.
Inserting \eqref{995} and \eqref{new3.12} into \eqref{3-22}, and letting $p\to\infty$ first and then $\rho\to 0$, we get $p \nabla^m u_{1,p}(x) \to\sum_{i=1}^{N}\gamma_{1,i} \nabla_x^mG(x,x_i)$ uniformly in $\mathcal K$. Thus
$$
   pu_{1,p}(x) \to \sum_{i=1}^{N}\gamma_{1,i} G(x,x_i) \quad\text{in} \quad C_{loc}^1(\overline{\om}\backslash\s).
$$
Similarly, we have $pu_{2,p}(x) \rightarrow\sum_{i=1}^{N}\gamma_{2,i} G(x,x_i)$ in $C_{loc}^1(\overline{\om}\backslash\s)$. Then by standard elliptic estimates, the convergence is also in $C_{loc}^2(\overline{\om}\backslash\s)$.
\end{proof}

Recall \eqref{fc-x1} that $x_1\in \mathcal{S}\cap\Omega$. Next, we prove that there is no boundary blow-up.
\begin{Lemma}\label{lemma-3-134}
    $\mathcal{S}\cap\partial\Omega=\emptyset$ and so $\mathcal{S}\subset\Omega$.
\end{Lemma}
\begin{proof} The proof follows the idea from \cite[Lemma 4.3]{wei}.
    Assume by contradiction that $\mathcal{S}\cap\partial\Omega\neq\emptyset$, say $x_{N}\in \mathcal{S}\cap \partial\Omega$ for example. Note from \eqref{fc-3-57} that $B_{2r_0}(x_N)\cap\mathcal{S}=\{x_N\}$, and recall that $x_N=\lim_{p\rightarrow\infty}x_{p,j}$ for some $2\leq j\leq n_0$. 
    
  Recalling $(\mathcal{P}_{2}^{n_0})$,  if Type $\mathcal{A}$  happens for $x_{p,j}$, similarly as Remark \ref{remark3-1}, we have that for any $r\in(0,r_0)$, 
    \begin{align}
        &\liminf_{p\to\infty}p\int_{B_{r}(x_N)\cap\Omega} \left(\mu_{k} u_{k,p}^{p+1} +\beta u_{1,p}^{\frac{p+1}{2}}u_{2,p}^{\frac{p+1}{2}}\right) dx\nonumber\\
        \geq &  \int_{\mathbb{R}^2}\left(\mu_{k} e^{V_{k}} +\beta e^{\frac{V_{1}+V_{2}}{2}}\right)dz=8\pi,\quad k=1,2.\nonumber
    \end{align}
    If Type $\mathcal{B}$ or Type $\mathcal{C}$ happens for $x_{p,j}$, similarly we have that for any $r\in(0,r_0)$,
    \begin{align*}
        \liminf_{p\to\infty}p\int_{B_{r}(x_N)\cap\Omega} \mu_{k} u_{k,p}^{p+1}dx
        \geq \int_{\mathbb{R}^2}\mu_{k} e^{\widetilde V_{k}} dz= 8\pi\quad\text{for some }\,k\in\{1,2\}.
    \end{align*}
    Therefore, we always have that for any $r\in(0,r_0)$,
    \begin{align}\label{eq-3-93}
         \liminf_{p\to\infty}p\int_{B_{r}(x_N)\cap\Omega} \left(\mu_1 u_{1,p}^{p+1}+\mu_2 u_{2,p}^{p+1} +2\beta u_{1,p}^{\frac{p+1}{2}}u_{2,p}^{\frac{p+1}{2}}\right) dx\geq 8\pi.
    \end{align}
    
    Let $\Vec{n}(x)$ be the outer normal vector of $(\partial B_{r}(x_N)\cap\Omega)\cup(\partial\Omega\cap B_{r}(x_N))$ and define
    $y_p:=x_N+\rho_{r,p}\Vec{n}(x_N)$, where
    \begin{align}
        \rho_{r,p}:=\frac{\int_{\partial\Omega\cap B_{r}(x_N)}\langle x-x_N, \Vec{n}(x)\rangle (|\nabla u_{1,p}|^2+|\nabla u_{2,p}|^2) dS_x}{\int_{\partial\Omega\cap B_{r}(x_N)} \langle \Vec{n}(x_N),\Vec{n}(x)\rangle (|\nabla u_{1,p}|^2+|\nabla u_{2,p}|^2) dS_x},
    \end{align}
    which implies that
    \begin{align}\label{eq-3-95}
        \sum_{k=1}^2\int_{\partial\Omega\cap B_{r}(x_N)}|\nabla u_{k,p}|^2 \langle x-y_p, \Vec{n}(x)\rangle dS_x=0.
    \end{align}
    We can take $r >0$ small enough such that $\langle \Vec{n}(x_1),\Vec{n}(x)\rangle\in [\frac{1}{2},1]$ for any $x\in\partial\Omega\cap B_{r}(x_N)$. Then $|\rho_{r,p}|\leq 2r$.
    Note from $u_{k,p}=0$ on $\partial\Omega$ that
    \begin{align*}
        \nabla u_{k,p}(x)=-|\nabla u_{k,p}(x)|\Vec{n}(x)\quad\text{ for }\,\,\forall x\in\partial\Omega\cap B_{r}(x_N).
    \end{align*}
    which shows that
    \begin{align}\label{eq-3-97}
    &\sum_{k=1}^2\int_{\partial\Omega\cap B_{r}(x_N)}\langle \nabla u_{k,p},x-y_p\rangle \langle \nabla u_{k,p}, \Vec{n}(x)\rangle dS_x\\
        =&\sum_{k=1}^2\int_{\partial\Omega\cap B_{r}(x_N)}|\nabla u_{k,p}|^2 \langle x-y_p, \Vec{n}(x)\rangle dS_x=0.\nonumber
      \end{align}
    Applying the Pohozaev identity \eqref{eq-2--7} with $\Omega'=B_{r}(x_N)\cap\Omega$ and $y=y_p$, and using \eqref{eq-3-95} and \eqref{eq-3-97}, we obtain
    \begin{align}\label{eq-3-98}
        &\frac{2p^2}{p+1}\int_{B_{r}(x_N)\cap\Omega}\left(\mu_1 u_{1,p}^{p+1}+\mu_2 u_{2,p}^{p+1}+2\beta u_{1,p}^{\frac{p+1}{2}}u_{2,p}^{\frac{p+1}{2}}\right) dx\\
        =&\sum_{k=1}^2\int_{\partial B_{r}(x_N)\cap\Omega}\langle p\nabla u_{k,p},x-y_p\rangle \langle p\nabla u_{k,p}, \Vec{n}(x)\rangle dS_x\nonumber\\
        &-\frac{1}{2}\sum_{k=1}^2\int_{\partial B_{r}(x_N)\cap\Omega}|p\nabla u_{k,p}|^2 \langle x-y_p, \Vec{n}(x)\rangle dS_x\nonumber\\
        &+\frac{p^2}{p+1}\int_{\partial B_{r}(x_N)\cap\Omega}\langle x-y_p, \Vec{n}(x)\rangle\left(\mu_1 u_{1,p}^{p+1}+\mu_2 u_{2,p}^{p+1}+2\beta u_{1,p}^{\frac{p+1}{2}}u_{2,p}^{\frac{p+1}{2}}\right)dS_x\nonumber\\
        =&I_p+II_p+III_p.\nonumber
     \end{align}
     
     Noting from $|\rho_{r,p}|\leq 2r$ that $|x-y_p|\leq 3r$ for any $x\in \partial B_r(x_N)$, we see from \eqref{31-1} that
     \[\big|III_p\big|\leq C\frac{C_r^{p+1}}{p^p}\to 0\quad\text{as }\;p\to\infty.\]
     Up to a subsequence, we may assume $y_p\to \tilde y_r$ with $|\tilde y_r-x_N|\leq 2r$. Recall Lemma \ref{lemma-3-7} that \[p u_{k,p}\rightarrow F_k:=\sum\limits_{i=1}^N \gamma_{k,i} G(\cdot, x_i)\quad\text{ in }\,\,C^{2}_{loc}(\overline{\Omega\cap B_{r}(x_N)}\setminus\{x_N\}).\] 
Since $x_N\in\partial\Omega$ implies $G(x,x_N)\equiv 0$ for $x\neq x_N$, we have (see e.g. \cite[(3.7)]{D-I-P-17}) 
\[F_k(x)=O(1),\quad \nabla F_k(x)=O(1),\quad \text{\it for }x\in \overline{\Omega\cap B_{r_0}(x_N)}\setminus\{x_N\}.\]
From here and $0<r<r_0$, we obtain
	\begin{align*}\lim_{p\to\infty}I_p
	=&\sum_{k=1}^2\int_{\partial B_{r}(x_N)\cap\Omega}\langle \nabla F_k,x-\tilde y_r\rangle \langle \nabla F_k, \Vec{n}(x)\rangle dS_x\\
	=& O(1) \int_{\partial B_{r}(x_N)\cap\Omega}|x-\tilde y_r| d S_x=O(r^2),\end{align*}
and similarly	$$\lim_{p\to\infty}II_p=O(r^2).$$
    Inserting these estimates into \eqref{eq-3-98}, we finally obtain
    \begin{align*}
\lim\limits_{r\rightarrow 0}\lim\limits_{p\rightarrow\infty}\frac{2p^2}{p+1}\int_{B_{r}(x_N)\cap\Omega}\left(\mu_1 u_{1,p}^{p+1}+\mu_2 u_{2,p}^{p+1}+2\beta u_{1,p}^{\frac{p+1}{2}}u_{2,p}^{\frac{p+1}{2}}\right) dx=0,
    \end{align*}
    which is a contradiction with \eqref{eq-3-93}.
    \end{proof}
    
Thanks to Lemma \ref{lemma-3-134}, by taking $r_0$ smaller if necessary, 
    we can further assume that for any $x_i\neq x_j\in\mathcal{S}$,
    \begin{align}\label{xixj}
         B_{2r_0}(x_i)\cap B_{2r_0}(x_j)=\emptyset, \,\, B_{2r_0}(x_i)\subset\Omega \,\,\text{ and } \,\,B_{2r_0}(x_j)\subset\Omega. 
    \end{align}

    Fix any $x_i\in\mathcal{S}$. Recalling  $M_p=\max\{u_{1,p}, u_{2,p}\}$, we take $y_{p,i}\in \overline{B_{r_0}(x_i)}$ such that
    \begin{align}\label{eq-3-106}
        M_p(y_{p,i}):=\max\limits_{x\in\overline{B_{r_0}(x_i)}}M_p(x).
    \end{align}
    Define
    \begin{align}\label{eq-3-107}
         v_{k,p,i}^*(z):=\frac{p}{M_p(y_{p,i})}\left(u_{k,p}(\varepsilon_{p,i}^*z+y_{p,i})-M_p(y_{p,i})\right),
        \end{align} 
    where $$\varepsilon_{p,i}^*:=(p M_p(y_{p,i})^{p-1})^{-\frac{1}{2}}.$$ 
    Clearly there is $m=m_i\in\{1,\ldots,n_0\}$ such that
    \begin{equation}
 \label{i-mi} R_{p,n_0}(y_{p,i})=\min_{1\leq j\leq n_0}|y_{p,i}-x_{p,j}|=|y_{p,i}-x_{p,m}|.
 \end{equation}   
This, together with \eqref{xixj} and $y_{p,i}\in \overline{B_{r_0}(x_i)}$, implies $\lim_{p\rightarrow\infty}x_{p,m}=x_i$. The next result shows that $\lim_{p\rightarrow\infty}y_{p,i}=x_i$ and the sequence $(x_{p,m})_{p}$ in Proposition \ref{prop-3-6} can be replaced with the sequence $(y_{p,i})_{p}$ of local maximum points near the concentration point $x_i$.

\begin{Lemma}\label{lemma-3-10} Under the above notations, we have $\lim_{p\rightarrow\infty}y_{p,i}=x_i$.  Besides,
    take $m=m_i\in\{1,\ldots,n_0\}$ such that \eqref{i-mi} holds and so $\lim_{p\rightarrow\infty}x_{p,m}=x_i$. Suppose $(v_{1,p,m},v_{2,p,m})$ satisfies type $\mathcal{A}$ (resp. type $\mathcal{B}$ or type $\mathcal{C}$) in $(\mathcal{P}_2^{n_0})$.
    Then $(v_{1,p,i}^*,v_{2,p,i}^*)$ satisfies the same type $\mathcal{A}$  (resp. type $\mathcal{B}$ or type $\mathcal{C}$). 
\end{Lemma}
\begin{proof}
 Since \eqref{eq-3-106} implies $M_p(y_{p,i})\geq M_p(x_{p,m})>0$, we have
\begin{align}
    0<\varepsilon_{p,i}^*=\left(p M_p(y_{p,i})^{p-1}\right)^{-\frac{1}{2}}\leq \left(p M_p(x_{p,m})^{p-1}\right)^{-\frac{1}{2}}=\varepsilon_{p,m}\to 0.
\end{align}
This, together with $(\p_3^{n_0})$, shows that
\begin{align}\label{fc-3-75}
    \frac{|x_{p,m}-y_{p,i}|}{\varepsilon_{p,m}}=&\frac{|x_{p,m}-y_{p,i}|}{\varepsilon_{p,i}^*}\frac{\varepsilon_{p,i}^*}{\varepsilon_{p,m}}\leq \frac{|x_{p,m}-y_{p,i}|}{\varepsilon_{p,i}^*}\\
    = &\left(p R_{p,n_0} (y_{p,i})^2 M_p(y_{p,i})^{p-1}\right)^{\frac{1}{2}}\leq C.\nonumber
\end{align}
Thus
\begin{align}
  y_{p,i}\rightarrow x_i\;\text{and}\;  \frac{B_{r_0}(x_i)-y_{p,i}}{\varepsilon_{p,i}^*}\rightarrow \mathbb{R}^2,\quad\text{as }p\to\infty.
\end{align}
Furthermore, up to a subsequence, we may assume $$\frac{y_{p,i}-x_{p,m}}{\varepsilon_{p,m}}\to z_\infty\in\mathbb R^2.$$

\textbf{Case 1.} 
Assuming $(v_{1,p,m},v_{2,p,m})$ satisfies type $\mathcal{A}$, we prove that $(v_{1,p,i}^*,v_{2,p,i}^*)$ also satisfies type $\mathcal{A}$.

By $(\p_2^{n_0})$, we have
\begin{align}
    v_{k,p,m}\left(\frac{y_{p,i}-x_{p,m}}{\varepsilon_{p,m}}\right)\rightarrow V_{k}(z_{\infty})\leq 0,\quad k=1,2,
\end{align}
which shows that
\begin{align}\label{eq-3-113}
    1\leq & \left(\frac{\varepsilon_{p,m}}{\varepsilon_{p,i}^*}\right)^2 = \left(\frac{M_p(y_{p,i})}{M_p(x_{p,m})}\right)^{p-1}
    = \max\left\{\left(1+\frac{ v_{1,p,m}\left(\frac{y_{p,i}-x_{p,m}}{\varepsilon_{p,m}}\right)}{p}\right)^{p-1},\right.\\
    &\left.\left(1+\frac{ v_{2,p,m}\left(\frac{y_{p,i}-x_{p,m}}{\varepsilon_{p,m}}\right)}{p}\right)^{p-1}\right\}\rightarrow \max \{e^{V_{1}(z_{\infty})},e^{V_{2}(z_{\infty})}\}\leq 1.\nonumber
\end{align}
This implies that $z_\infty=0$,
\begin{align}
    \frac{y_{p,i}-x_{p,m}}{\varepsilon_{p,m}}\to 0,\quad \frac{\varepsilon_{p,i}^*}{\varepsilon_{p,m}}\nearrow 1,\quad\frac{M_p(x_{p,m})}{M_p(y_{p,i})}\nearrow 1.\label{eq-3-114}
\end{align}
We claim that
\begin{align}
    p\left(\frac{M_p(x_{p,m})}{M_p(y_{p,i})}-1\right)\to 0.\label{eq-3-114^*}
\end{align}
If not, then there exists $C_1< 0$ such that up to a subsequence,
\begin{align*}
    p\left(\frac{M_p(x_{p,m})}{M_p(y_{p,i})}-1\right)\leq C_1.
\end{align*}
It follows that
\begin{align*}
    1\leftarrow \left(\frac{\varepsilon_{p,i}^*}{\varepsilon_{p,m}}\right)^2=\left(\frac{M_p(x_{p,m})}{M_p(y_{p,i})}\right)^{p-1}\leq\left(1+\frac{C_1}{p}\right)^{p-1}\to e^{C_1}<1,
\end{align*}
which is a contradiction. 

Recalling  the definition \eqref{eq-3-107} of $v_{k,p,i}^*$ and the definition \eqref{fc-3-vkp} of $v_{k,p,m}$, a direct computation gives
\begin{align}\label{fc-3-81}
v_{k,p,i}^*(z)=\frac{M_p(x_{p,m})}{M_p(y_{p,i})}v_{k,p,i}\left(\frac{\varepsilon_{p,i}^*}{\varepsilon_{p,m}}z+\frac{y_{p,i}-x_{p,m}}{\varepsilon_{p,m}}\right)+ p\left(\frac{M_p(x_{p,m})}{M_p(y_{p,i})}-1\right).
\end{align}
From here,  \eqref{eq-3-114}-\eqref{eq-3-114^*}, and $(\p_2^{n_0})$ which says that $(v_{1,p,m},v_{2,p,m})\rightarrow ({V}_{1},{V}_{2})$ in $C^2_{loc}(\mathbb{R}^2)\times C^2_{loc}(\mathbb{R}^2)$, we conclude that $(v_{1,p,i}^*,v_{2,p,i}^*)\rightarrow ({V}_{1},{V}_{2})$ in $C^2_{loc}(\mathbb{R}^2)\times C^2_{loc}(\mathbb{R}^2)$.

\textbf{Case 2.} Assuming $(v_{1,p,m},v_{2,p,m})$ satisfies type $\mathcal{B}$ (the case type $\mathcal{C}$ is similar), we prove that $(v_{1,p,i}^*,v_{2,p,i}^*)$ satisfies the same type $\mathcal{B}$. 

By $(\p_2^{n_0})$, we have
\begin{align}
    v_{1,p,m}\left(\frac{y_{p,i}-x_{p,m}}{\varepsilon_{p,m}}\right)\rightarrow \widetilde V_{1}(z_{\infty})\leq 0,\quad v_{2,p,m}\left(\frac{y_{p,i}-x_{p,m}}{\varepsilon_{p,m}}\right)\rightarrow -\infty.
\end{align}
Then the same argument as \eqref{eq-3-113} implies that \eqref{eq-3-114}-\eqref{eq-3-114^*} also holds. Therefore, it follows from \eqref{fc-3-81} and $(\p_2^{n_0})$ that
         $v_{1,p,i}^*\rightarrow \widetilde V_{1}$ in $C^{2}_{loc}(\mathbb{R}^2)$ and 
        and $v_{2,p,i}^*\rightarrow -\infty$ uniformly in  any compact subsets of $\mathbb{R}^2$.
         This completes the proof.
\end{proof}

We are ready to prove Theorem \ref{th-1---1}.

\begin{proof}[Proof of Theorem \ref{th-1---1}]  Clearly Theorem \ref{th-1---1} follows from Proposition \ref{prop-3-6} and Lemma \ref{lemma-3-10}.
\end{proof}

\section{Refined estimates around the concentration set}
\label{section-4}
In this section, we fix any $x_i\in\mathcal{S}$ and give the estimate of the local maximums $\Vert u_{k,p}\Vert_{L^{\infty}(B_{r_0}(x_i))}$ near the concentration point $x_i$. In particular, we will prove $\lim\limits_{p\to\infty}\Vert u_{k,p}\Vert_{L^{\infty}(B_{r_0}(x_i))}\in \{0,\sqrt{e}\}$ and Theorem \ref{thm-localmass}.

\begin{Lemma}\label{Lemma-4-1++}
    Recall $\gamma_{k,i}$ defined in Lemma \ref{lemma-3-7}. Then for any $\rho\in (0,r_0)$,
    \begin{align}\label{eq-4-1^*}
       \lim\limits_{p\rightarrow \infty} p\int_{B_{\rho}(x_i)}\left(\mu_1 u_{1,p}^{p+1}+\mu_2 u_{2,p}^{p+1}+2\beta u_{1,p}^{\frac{p+1}{2}}u_{2,p}^{\frac{p+1}{2}} \right)dx=\frac{1}{8\pi}\left(\gamma_{1,i}^2+\gamma_{2,i}^2\right).
    \end{align}
\end{Lemma}
\begin{proof}
First, for any $0<\rho_1<\rho_2<r_0$, it follows from \eqref{31-1} that
\[p\int_{B_{\rho_2}(x_i)\setminus B_{\rho_1}(x_i)}\left(\mu_1 u_{1,p}^{p+1}+\mu_2 u_{2,p}^{p+1}+2\beta u_{1,p}^{\frac{p+1}{2}}u_{2,p}^{\frac{p+1}{2}} \right)dx\leq C\frac{C^{p+1}}{p^p}\to 0,\]
so the left hand side of \eqref{eq-4-1^*} is independent of the choice of $\rho$, namely
\begin{align}\label{fc-4-01}\lim_{\rho\to 0} \lim\limits_{p\rightarrow \infty} p\int_{B_{\rho}(x_i)}\left(\mu_1 u_{1,p}^{p+1}+\mu_2 u_{2,p}^{p+1}+2\beta u_{1,p}^{\frac{p+1}{2}}u_{2,p}^{\frac{p+1}{2}} \right)dx\\
=\lim\limits_{p\rightarrow \infty} p\int_{B_{\rho}(x_i)}\left(\mu_1 u_{1,p}^{p+1}+\mu_2 u_{2,p}^{p+1}+2\beta u_{1,p}^{\frac{p+1}{2}}u_{2,p}^{\frac{p+1}{2}} \right)dx.\nonumber\end{align}

    Applying the Pohozaev identity \eqref{eq-2--7} with $\Omega'=B_{\rho}(x_i)$ and $y=x_i$, we have
    \begin{align}\label{eq-4-2-1}
        &\frac{2 p^2}{p+1}\int_{B_{\rho}(x_i)}\left(\mu_1 u_{1,p}^{p+1}+\mu_2 u_{2,p}^{p+1}+2\beta u_{1,p}^{\frac{p+1}{2}}u_{2,p}^{\frac{p+1}{2}}\right) dx\nonumber\\
        =&\rho\sum_{k=1}^2\int_{\partial B_{\rho}(x_i)} \langle p\nabla u_{k,p}, \Vec{n}(x)\rangle ^2 dS_x
        -\frac{\rho}{2}\sum_{k=1}^2\int_{\partial B_{\rho}(x_i)}|p\nabla u_{k,p}|^2dS_x\\
        &+\frac{\rho p^2}{p+1}\int_{\partial B_{\rho}(x_i)}\left(\mu_1 u_{1,p}^{p+1}+\mu_2 u_{2,p}^{p+1}+2\beta u_{1,p}^{\frac{p+1}{2}}u_{2,p}^{\frac{p+1}{2}}\right)dS_x\nonumber,
    \end{align}
    Again by \eqref{31-1}, we get
    \begin{align}
        \frac{\rho p^2}{p+1}\int_{\partial B_{\rho}(x_i)}\left(\mu_1 u_{1,p}^{p+1}+\mu_2 u_{2,p}^{p+1}+2\beta u_{1,p}^{\frac{p+1}{2}}u_{2,p}^{\frac{p+1}{2}}\right)dS_x\to 0.\label{eq-4-3-1}
    \end{align}
    By Lemma \ref{lemma-3-7}, we have that for $x\in\partial B_{\rho}(x_i)$ and $k=1,2$,
    \begin{align*}
        p\nabla u_{k,p}=-\frac{\gamma_{k,i}}{2\pi}\frac{x-x_i}{|x-x_i|^2}+O(1),
    \end{align*}
    which implies that
    \begin{align}
        &\rho\int_{\partial B_{\rho}(x_i)}\langle p\nabla u_{k,p}, \Vec{n}(x)\rangle ^2 dS_x-\frac{\rho}{2}\int_{\partial B_{\rho}(x_i)}|p\nabla u_{k,p}|^2 dS_x=\frac{\gamma_{k,i}^2}{4\pi}+O(\rho).\label{eq-4-6-1}
    \end{align}
    Inserting \eqref{eq-4-3-1} and \eqref{eq-4-6-1} into \eqref{eq-4-2-1}, letting $p\to\infty$ first and then $\rho\to 0$, we obtain
    \eqref{eq-4-1^*}.
\end{proof}

Define the ``local mass'' contributed by the concentration point $x_i$:
\begin{align}
    \sigma_{k,i}:=\lim\limits_{p\rightarrow \infty}  \frac{p}{M_p(y_{p,i})} \int_{B_{\rho}(x_{i})} \left(\mu_k u_{k,p}^p +\beta u_{k,p}^{\frac{p-1}{2}}u_{3-k,p}^{\frac{p+1}{2}}\right) dx,\quad k=1,2,\label{eq-4-9-}
\end{align}
where $\rho\in (0,r_0)$.
The same argument as \eqref{fc-4-01} implies that
$\sigma_{k,i}$ is independent of the choice of $\rho$, so
\begin{align}
    \sigma_{k,i}=\lim_{\rho\to 0}\lim\limits_{p\rightarrow \infty}  \frac{p}{M_p(y_{p,i})} \int_{B_{\rho}(x_{i})} \left(\mu_k u_{k,p}^p +\beta u_{k,p}^{\frac{p-1}{2}}u_{3-k,p}^{\frac{p+1}{2}}\right) dx=\frac{\gamma_{k,i}}{\lim\limits_{p\to\infty}M_p(y_{p,i})}.\label{eq-4-9-1}
\end{align}

\begin{corollary}\label{cor-4-2} We have
    \begin{align}\label{eq-c-4--2}
        \sigma_{1,i}^2+\sigma_{2,i}^2\leq 8\pi(\sigma_{1,i}+\sigma_{2,i}).
    \end{align}
\end{corollary}
\begin{proof} 
   Since $M_p(y_{p,i})=\max\limits_{x\in\overline{B_{r_0}(x_i)}}M_p(x)$, we have
    \begin{align}
        &p\int_{B_{\rho}(x_i)}\mu_1 u_{1,p}^{p+1}+\mu_2 u_{2,p}^{p+1}+2\beta u_{1,p}^{\frac{p+1}{2}}u_{2,p}^{\frac{p+1}{2}} \,dx\nonumber\\
        \leq & M_p(y_{p,i}) \sum_{k=1}^2p\int_{B_{\rho}(x_i)} \mu_k u_{k,p}^p +\beta u_{k,p}^{\frac{p-1}{2}}u_{3-k,p}^{\frac{p+1}{2}} dx,\nonumber
    \end{align}
    which together with \eqref{eq-4-1^*} shows that
    \begin{align}\label{eq-4-5^*}
        &\frac1{8\pi}(\gamma_{1,i}^2+\gamma_{2,i}^2)
        \leq (\gamma_{1,i}+\gamma_{2,i})\lim_{p\rightarrow\infty}M_p(y_{p,i}) .
    \end{align}
    Inserting $\gamma_{k,i}=\sigma_{k,i}\lim\limits_{p\rightarrow\infty}M_p(y_{p,i})$ into \eqref{eq-4-5^*} and using $1\leq \lim\limits_{p\rightarrow\infty}M_p(y_{p,i})<+\infty$, we obtain \eqref{eq-c-4--2}.
\end{proof}

To proceed, we need to discuss type $\mathcal{A}$ and type $\mathcal{B}$ or type $\mathcal{C}$ separately.

\subsection{The type $\mathcal{A}$ case.}
\label{section-4-1}
Let us assume that $(v_{1,p,i}^*,v_{2,p,i}^*)$ satisfies type $\mathcal{A}$. 
\begin{Lemma}\label{lemma-4-3++++++}
    $\sigma_{1,i}=\sigma_{2,i}=8\pi$.
\end{Lemma}
\begin{proof}
Similarly as Remark \ref{remark3-1}, we have 
    \begin{align}
       \sigma_{k,i}= &\lim\limits_{p\rightarrow \infty}  \frac{p}{M_p(y_{p,i})} \int_{B_{\rho}(x_{i})} \left(\mu_k u_{k,p}^p +\beta u_{k,p}^{\frac{p-1}{2}}u_{3-k,p}^{\frac{p+1}{2}}\right) dx\nonumber\\
        \geq &  \int_{\mathbb{R}^2}\left(\mu_{k} e^{V_{k}} +\beta e^{\frac{V_{1}+V_{2}}{2}}\right)dz=8\pi,\quad k=1,2.\nonumber
    \end{align}
From here and Corollary \ref{cor-4-2}, we obtain  $\sigma_{1,i}=\sigma_{2,i}=8\pi$.
\end{proof}
We now prove a decay estimate of $v_{1,p,i}^*$ and $v_{2,p,i}^*$.
\begin{Lemma}\label{lemma-4-4}
For any $\eta\in (0,1]$ and $k=1,2$, there exist large $R_{\eta}>1$, $p_{\eta}>p_0$ and $C_{\eta}>0$ such that
\begin{align}\label{eq-4-18^*}
    v_{k,p,i}^*(z)\leq -\left(8-\frac{\sigma_{k,i}}{2\pi}-\eta\right)\log|z|+C_{\eta}= -\left(4-\eta\right)\log|z|+C_{\eta}
\end{align}
for any $2R_{\eta}\leq |z|\leq \frac{r_0}{3\varepsilon_{p,i}^*}$ and $p>p_{\eta}$.
\end{Lemma}

\begin{proof}
Take $k=1$ for example.
 There exists large $R_{\eta}>1$ such that
\begin{align}
    \int_{B_{R_{\eta}}(0)} \mu_1 e^{V_{1}} +\beta e^{\frac{{V}_{1}+{V}_{2}}{2}} dz\geq 8\pi-\frac{\eta}{2}.
\end{align} Denote
\begin{align}\label{fc-gp}
    g_p(z):=\mu_1\left(1+\frac{v_{1,p,i}^*(z)}{p}\right)^p+\beta\left(1+\frac{v_{1,p,i}^*(z)}{p}\right)^{\frac{p-1}{2}}\left(1+\frac{v_{2,p,i}^*(z)}{p}\right)^{\frac{p+1}{2}}.
\end{align}
Then by $(v_{1,p,i}^*, v_{2,p,i}^*)\to (V_1, V_2)$ in $C_{loc}^2(\mathbb R^2)\times C_{loc}^2(\mathbb R^2)$, there exists large $p_{\eta}>1$ such that for all $p>p_{\eta}$,
\begin{align}\label{eq-4-18}
    &\int_{B_{R_{\eta}}(0)} g_p(z) dz\geq \int_{B_{R_{\eta}}(0)} \mu_1 e^{V_{1}} +\beta e^{\frac{{V}_{1}+{V}_{2}}{2}} dz-\frac{\eta}{2}\geq 8\pi-\eta.
\end{align}
Furthermore, letting $\delta=r_0/3$, it follows from \eqref{31-1} and \eqref{eq-4-9-} that
\begin{align*}
\lim_{p\to\infty}\int_{B_{\frac{2\delta}{\varepsilon_{p,i}^*}}(0)}g_p(w)dw=\lim_{p\to\infty}\frac{p}{M_p(y_{p,i})}\int_{B_{2\delta}(y_{p,i})}
\mu_{1}u_{1,p}^{p}+\beta u_{1,p}^{\frac{p-1}{2}}u_{2,p}^{\frac{p+1}{2}} dx=\sigma_{1,i},\end{align*}
so by taking $p_{\eta}$ larger if necessary, we have that for all $p>p_{\eta}$,
\begin{align}\label{eq-4-18-0}\int_{B_{\frac{2\delta}{\varepsilon_{p,i}^*}}(0)}g_p(w)dw\leq \sigma_{1,i}+\eta.\end{align}

By the Green's representation formula, we have
\begin{align*}
    u_{1,p}(\varepsilon_{p,i}^* z+y_{p,i})=\int_{\Omega} G(\varepsilon_{p,i}^* z+y_{p,i},y) \left(\mu_1 u_{1,p}^p +\beta u_{1,p}^{\frac{p-1}{2}} u_{2,p}^{\frac{p+1}{2}}\right) dy.
\end{align*}
Denoting $\widetilde{\Omega}_{p,i}:=\frac{\Omega-y_{p,i}}{\varepsilon_{p,i}^*}$, it follows from \eqref{eq-3-107} that
\begin{align}
    v_{1,p,i}^*(z)=-p+\int_{\widetilde{\Omega}_{p,i}} G(\varepsilon_{p,i}^* z+y_{p,i},\varepsilon_{p,i}^* w+y_{p,i}) g_p(w)dw.
\end{align}
From here and \eqref{eq-1-6^*}, we have
\begin{align}\label{eq-4-22}
    &v_{1,p,i}^*(z)-v_{1,p,i}^*(0)\nonumber\\
    =&\int_{\widetilde{\Omega}_{p,i}} \left[G(\varepsilon_{p,i}^* z+y_{p,i},\varepsilon_{p,i}^* w+y_{p,i})-G(y_{p,i},\varepsilon_{p,i}^* w+y_{p,i})\right] g_p(w) dw\\
    =&\frac{1}{2\pi}\int_{B_{\frac{2\delta}{\varepsilon_{p,i}^*}}(0)} \log{\frac{|w|}{|w-z|}} g_p(w) dw+\frac{1}{2\pi}\int_{\widetilde{\Omega}_{p,i}\setminus B_{\frac{2\delta}{\varepsilon_{p,i}^*}}(0)} \log{\frac{|w|}{|w-z|}} g_p(w) dw\nonumber \\
    &-\int_{\widetilde{\Omega}_{p,i}} \left[H(\varepsilon_{p,i}^* z+y_{p,i},\varepsilon_{p,i}^* w+y_{p,i})-H(y_{p,i},\varepsilon_{p,i}^* w+y_{p,i})\right] g_p(w) dw\nonumber\\
    =&I(z)+II(z)+III(z).\nonumber
\end{align}

Now we let $2R_{\eta}\leq |z|\leq \frac{\delta}{\varepsilon_{p,i}^*}$ and $p>p_{\eta}$.
Since $H(\cdot,\cdot)$ is bounded in $B_{r_0}(x_i)\times\Omega$, it follows that
\begin{align*}
    III(z)&=O(1)\int_{\widetilde{\Omega}_{p,i}}g_p(w) dw\\
    &=O(1)\frac{p}{M_p(y_{p,i})}\int_{\Omega}\mu_{1}u_{1,p}^{p}+\beta u_{1,p}^{\frac{p-1}{2}}u_{2,p}^{\frac{p+1}{2}} dx=O(1).\nonumber
\end{align*}
For $|z|\leq {\delta}/{\varepsilon_{p,i}^*}$ and $|w|\geq {2\delta}/{\varepsilon_{p,i}^*}$, we have $\frac{2}{3}\leq \frac{|w|}{|w-z|}\leq 2$, which shows that
\begin{align*}
    II(z)=O(1)\int_{\widetilde{\Omega}_{p,i}\setminus B_{{2\delta}/{\varepsilon_{p,i}^*}}(0)}g_p(w) dw=O(1).
\end{align*}
Next, we divide $I(z)$ into the following parts:
\begin{align}
    I(z)=&\frac{1}{2\pi}\int_{B_{R_{\eta}}(0)} \log{\frac{|w|}{|w-z|}} g_p(w) dw\\
    &+\frac{1}{2\pi}\int_{B_{{2\delta}/{\varepsilon_{p,i}^*}}(0)\setminus B_{R_{\eta}}(0) \cap\{|w|\leq 2|z-w|\}} \log{\frac{|w|}{|w-z|}} g_p(w) dw\nonumber\\
    &+\frac{1}{2\pi}\int_{B_{{2\delta}/{\varepsilon_{p,i}^*}}(0)\setminus B_{R_{\eta}}(0)\cap\{|w|\geq 2|z-w|\}} \log{|w|} g_p(w) dw\nonumber\\
    &+\frac{1}{2\pi}\int_{B_{{2\delta}/{\varepsilon_{p,i}^*}}(0)\setminus B_{R_{\eta}}(0) \cap\{|w|\geq 2|z-w|\}} \log{\frac{1}{|w-z|}} g_p(w) dw\nonumber\\
    =&I_1(z)+I_2(z)+I_3(z)+I_4(z).\nonumber
\end{align}
For $|w|\leq R_{\eta}$ and $|z|\geq 2R_{\eta}$, we have $\frac{|w|}{|w-z|}\leq \frac{2R_{\eta}}{|z|}\leq 1$. This, together with \eqref{eq-4-18}, implies that
\begin{align}
    I_1(z)\leq& \frac{1}{2\pi} \log{\frac{2R_{\eta}}{|z|}}\int_{B_{R_{\eta}}(0)} g_p(w) dw
    \leq -\frac{1}{2\pi}\left(8\pi-\eta \right)\log|z|+C.
\end{align}
Also,
\begin{align*}
     I_2(z)\leq \frac{\log{2}}{2\pi} \int_{B_{{2\delta}/{\varepsilon_{p,i}^*}}(0)\setminus B_{R_{\eta}}(0)} g_p(z) dz\leq C.
\end{align*}
For $|w|\geq 2|z-w|$, we have $|w|\leq 2|z|$. So we see from \eqref{eq-4-18}-\eqref{eq-4-18-0} that
\begin{align*}
    I_3(z)\leq &\frac{1}{2\pi}\log(2|z|)\int_{B_{{2\delta}/{\varepsilon_{p,i}^*}}(0)\setminus B_{R_{\eta}}(0)}  g_p(w) dw\\
    \leq & \frac{1}{2\pi}(\sigma_{1,i}-8\pi+2\eta)\log|z|+C.\nonumber
\end{align*}
Since $1+\frac{v_{k,p,i}^*(w)}{p}=\frac{u_{k,p}(\varepsilon_{p,i}^*w+y_{p,i})}{M_p(y_{p,i})}\leq 1$ and so $0\leq g_p(w)\leq \mu_k+\beta$ for $|z|\leq \frac{2\delta}{\varepsilon_{p,i}^*}$,
we have
\begin{align*}
    I_4(z)= &\frac{1}{2\pi}\int_{B_{{2\delta}/{\varepsilon_{p,i}^*}}(0)\setminus B_{R_{\eta}}(0) \cap\{2\leq 2|z-w|\leq|w|\}} \log{\frac{1}{|w-z|}} g_p(w) dw\\
     &+\frac{1}{2\pi}\int_{B_{{2\delta}/{\varepsilon_{p,i}^*}}(0)\setminus B_{R_{\eta}}(0) \cap\{|z-w|\leq 1\}} \log{\frac{1}{|w-z|}} g_p(w) dw\nonumber\\
    \leq & C \int_{\{|z-w|\leq 1\}} \log{\frac{1}{|w-z|}} dw\leq C.\nonumber
\end{align*}
Inserting above estimates into \eqref{eq-4-22} and using $v_{1,p,i}^*(0)\to V_1(0)$, we obtain
$$v_{1,p,i}^*(z)\leq -\left(8-\frac{\sigma_{1,i}}{2\pi}-\frac{3\eta}{2\pi}\right)\log|z|+C_\eta$$
for any $2R_{\eta}\leq |z|\leq \frac{\delta}{\varepsilon_{p,i}^*}$ and $p>p_{\eta}$.
\end{proof}
\begin{remark}\label{remark-4-5}
Lemma \ref{lemma-4-4} shows that for $p$ large,
$$\left(1+\frac{v_{k,p,i}^*}{p}\right)^{p}= e^{p\log\left(1+\frac{v_{k,p,i}^*}{p}\right)}\leq e^{v_{k,p,i}^*}\leq \frac{C}{|z|^{4-\eta}},\;\text{for }2R_{\eta}\leq |z|\leq \frac{r_0}{3\varepsilon_{p,i}^*}.$$
Meanwhile, since $v_{k,p,i}\to V_k$ in $C_{loc}^2(\mathbb R^2)$, we have $(1+\frac{v_{k,p,i}^*}{p})^{p}\leq C$ for $|z|\leq 2R_{\eta}$ and $p$ large. Therefore, for $p$ large, 
    \begin{align}
    &\left(1+\frac{v_{k,p,i}^*}{p}\right)^{p}\leq \frac{C}{(1+|z|)^{4-\eta}},\quad \text{for }|z|\leq \frac{r_0}{3\varepsilon_{p,i}^*}.
    \end{align}
    Similarly, we have
    \begin{align}
    &\left(1+\frac{v_{k,p,i}^*}{p}\right)^{\frac{p-1}{2}}\left(1+\frac{v_{3-k,p,i}^*}{p}\right)^{\frac{p+1}{2}}\leq \frac{C}{(1+|z|)^{4-\eta}},\quad \text{for }|z|\leq \frac{r_0}{3\varepsilon_{p,i}^*}.
\end{align}
\end{remark}
\begin{Lemma}\label{Lemma-4-6} We have
    \begin{align}\label{fc-4-21}
\lim\limits_{p\rightarrow+\infty} u_{1,p}(y_{p,i})=\lim\limits_{p\rightarrow+\infty} u_{2,p}(y_{p,i})=\sqrt{e}.
    \end{align}
    As a consequence,
     \begin{align}\label{fc-4-21-1}
\lim\limits_{p\rightarrow+\infty}\Vert u_{1,p}\Vert_{L^\infty(B_{r_0}(x_i))}&=\lim\limits_{p\rightarrow+\infty}\Vert u_{2,p}\Vert_{L^\infty(B_{r_0}(x_i))}=\sqrt{e},
    \end{align}
and $\gamma_{1,i}=\gamma_{2,i}=8\pi \sqrt{e}$.
\end{Lemma}
\begin{proof}
Note that $v_{k,p,i}^*(0)=V_k(0)+o(1)=O(1)$,                                                                                                                                                                                                                                                                                                                                                                                                                                                                                                                                                                                                                                                                                                                                                                                                                                                                                                                                             
which, together with $u_{k,p}(y_{p,i})=M_p(y_p)(1+\frac{v_{k,p,i}*(0)}{p})$, implies that
\begin{align}\label{eq-4-35+-}
    u_{k,p}(y_{p,i})=M_p(y_{p,i})+O\left(\frac{1}{p}\right)\geq \frac{1}{2},\quad\text{for $p$ large.}
\end{align}
 By the Green's representation formula, we have
    \begin{align}
        u_{1,p}(y_{p,i})=&\int_{\Omega}G(y_{p,i},x)\left(\mu_1 u_{1,p}^p+\beta u_{1,p}^{\frac{p-1}{2}} u_{2,p}^{\frac{p+1}{2}} \right) dx\\
        =&\int_{B_{r_0/3}(y_{p,i})}G(y_{p,i},x)\left(\mu_1 u_{1,p}^p+\beta u_{1,p}^{\frac{p-1}{2}} u_{2,p}^{\frac{p+1}{2}} \right) dx+O\left(\frac{1}{p}\right)\nonumber,
    \end{align}
    where we have used that
    \begin{align*}
 &\int_{\Omega\setminus B_{r_0/3}(y_{p,i})}G(y_{p,i},x)\left(\mu_1 u_{1,p}^p+\beta u_{1,p}^{\frac{p-1}{2}} u_{2,p}^{\frac{p+1}{2}} \right) dx\\
= &O(1)\int_{\Omega} \mu_1 u_{1,p}^p+\beta u_{1,p}^{\frac{p-1}{2}} u_{2,p}^{\frac{p+1}{2}}  dx=O\left(\frac{1}{p}\right).
    \end{align*}
  Recalling $g_p$ defined in \eqref{fc-gp},  it follows from \eqref{eq-1-6^*} that
    \begin{align}\label{eq-4-37+-}
        u_{1,p}(y_{p,i})=&\frac{M_p(y_{p,i})}{p}\int_{B_{{r_0}/{3\varepsilon_{p,i}^*}}(0)}G(y_{p,i},\varepsilon_{p,i}^*z+y_{p,i})g_p(z) dz+O\left(\frac{1}{p}\right)\\
        =&-\frac{M_p(y_{p,i})}{p}\int_{B_{{r_0}/{3\varepsilon_{p,i}^*}}(0)}H(y_{p,i},\varepsilon_{p,i}^*z+y_{p,i})g_p(z) dz\nonumber\\
        &-\frac{M_p(y_{p,i})}{2\pi p}\int_{B_{{r_0}/{3\varepsilon_{p,i}^*}}(0)}(\log|z|) g_p(z) dz\nonumber\\
        &-\frac{M_p(y_{p,i})}{2\pi p}\log{\varepsilon_{p,i}^*}\int_{B_{{r_0}/{3\varepsilon_{p,i}^*}}(0)}g_p(z) dz+O\left(\frac{1}{p}\right).\nonumber
    \end{align}
From Remark \ref{remark-4-5} and the dominated convergence theorem, we obtain
\begin{align*}
&\lim\limits_{p\rightarrow+\infty}\int_{B_{{r_0}/{3\varepsilon_{p,i}^*}}(0)}g_p(z) dz=\int_{\mathbb{R}^2} \mu_1 e^{V_1} +\beta e^{\frac{V_1+V_2}{2}} dz=8\pi,
\end{align*}
\begin{align*}
\lim\limits_{p\rightarrow+\infty}\int_{B_{{r_0}/{3\varepsilon_{p,i}^*}}(0)}(\log|z|) g_p(z) dz=\int_{\mathbb{R}^2}(\log|z|)\left(  \mu_1 e^{V_1} +\beta e^{\frac{V_1+V_2}{2}} \right) dz=O(1).
\end{align*}
Furthermore,
\begin{align*}
&\int_{B_{{r_0}/{3\varepsilon_{p,i}^*}}(0)}H(y_{p,i},\varepsilon_{p,i}^*z+y_{p,i}) g_p(z) dz=O(1)\int_{B_{{r_0}/{3\varepsilon_{p,i}^*}}(0)}g_p(z) dz=O(1).
\end{align*}
These, together with \eqref{eq-4-35+-} and \eqref{eq-4-37+-}, imply
\begin{align}
    \frac{\log{\varepsilon_{p,i}^*}}{p}=-\frac{1}{4}+o(1).
\end{align}
Since $$\log{p}+2\log{\varepsilon_{p,i}^*}+(p-1)\log{M_p(y_{p,i})}=0,$$ we obtain
$M_p(y_{p,i})\to \sqrt{e}$ as $p\to\infty$. Consequently, \eqref{fc-4-21} follows from \eqref{eq-4-35+-}, and $\gamma_{k,i}=\sigma_{k,i}\lim M_p(y_{p,i})=8\pi\sqrt{e}$. Finally, \eqref{fc-4-21-1} follows from \eqref{fc-4-21} and
$$M_p(y_{p,i})=\max\limits_{x\in\overline{B_{r_0}(x_i)}}M_p(x).$$
The proof is complete.
\end{proof}

\subsection{ The type $\mathcal{B}$ or  type $\mathcal{C}$ case.}\label{section-4-2} This case is more complicated than the type $\mathcal{A}$ case, and different ideas are needed.
Without loss of generality, we may assume that $(v_{1,p,i}^*,v_{2,p,i}^*)$ satisfies type $\mathcal{B}$.
\begin{Lemma}\label{Lemma-4-8}
  $\sigma_{1,i}=8\pi$.
\end{Lemma}
\begin{proof}
Similarly as Remark \ref{remark3-1}, we have 
    \begin{align}
       \sigma_{1,i}= &\lim\limits_{p\rightarrow \infty}  \frac{p}{M_p(y_{p,i})} \int_{B_{\rho}(x_{i})} \left(\mu_1 u_{1,p}^p +\beta u_{1,p}^{\frac{p-1}{2}}u_{2,p}^{\frac{p+1}{2}}\right) dx\nonumber\\
        \geq &  \int_{\mathbb{R}^2}\mu_{1} e^{\widetilde V_{1}} dz=8\pi.\nonumber
    \end{align}
Then Corollary \ref{cor-4-2} implies $8\pi\leq\sigma_{1,i}\leq 4(1+\sqrt{2})\pi$. 
Consequently, the same proof as Lemma \ref{lemma-4-4} implies that for any $\eta\in(0,0.1]$, there exist large $R_{\eta}>1$, $p_{\eta}>1$ and $C_{\eta}>0$ such that
\begin{align}\label{eq-4-32++-}
    v_{1,p,i}^*(z)&\leq -\left(8-\frac{\sigma_{1,i}}{2\pi}-\eta\right)\log|z|+C_{\eta}\\
    &\leq -(6-2\sqrt{2}-\eta)\log|z|+C_{\eta}\leq -3\log|z|+C_{\eta},\nonumber
\end{align}
for any $2R_{\eta}\leq |z|\leq \frac{\delta}{\varepsilon_{p,i}^*}$ and $p>p_{\eta}$, where $\delta=r_0/3$.
Then similarly as Remark \ref{remark-4-5}, we have that for $m\in \{-1,0,1\}$ and $p$ large,
\begin{align}\label{eq-4-32++}
    \left(1+\frac{v_{1,p,i}^*}{p}\right)^{p+m}\leq \frac{C}{(1+|z|)^3},\quad\forall |z|\leq \frac{\delta}{\varepsilon_{p,i}^*}.
\end{align}

Fix any large $R>1$.
By type $\mathcal{B}$ in $(\p_2^{n_0})$ and the dominated convergence theorem, we have for any $m=-1,0,1$ and $p$ large enough,
\begin{align*}
&\int_{B_{R}(0)} \left(1+\frac{v_{2,p,i}^*}{p}\right)^{p+m} dz=o(1),\\
    &\int_{B_{R}(0)} \left(1+\frac{v_{1,p,i}^*}{p}\right)^{p+m} dz=  \int_{B_{R}(0)}  e^{\widetilde{V}_{1}} dz+o(1),\\
    &\int_{B_{\frac{\delta}{\varepsilon_{p,i}^*}}(0)} \left(1+\frac{v_{1,p,i}^*}{p}\right)^{p+m} dz=\int_{\mathbb{R}^2}  e^{\widetilde{V}_{1}} dz+o(1),
    \end{align*}
    and so
\begin{align}\label{fc-4-29}
    0<&\int_{B_{\frac{\delta}{\varepsilon_{p,i}^*}}(0)} \left(1+\frac{v_{1,p,i}^*}{p}\right)^{\frac{p-1}{2}}\left(1+\frac{v_{2,p,i}^*}{p}\right)^{\frac{p+1}{2}} dz\\
    \leq & \left(\int_{B_{R}(0)} \left(1+\frac{v_{1,p,i}^*}{p}\right)^{p-1} dz\right)^{\frac{1}{2}} \left(\int_{B_{R}(0)} \left(1+\frac{v_{2,p,i}^*}{p}\right)^{p+1} dz\right)^{\frac{1}{2}} \nonumber\\
    &+\left(\int_{B_{\frac{\delta}{\varepsilon_{p,i}^*}}(0)\setminus B_{R}(0)} \left(1+\frac{v_{1,p,i}^*}{p}\right)^{p-1} dz\right)^{\frac{1}{2}} \left(\frac{p}{M_p(y_{p,i})^2}\int_{\Omega} u_{2,p}^{p+1} dx\right)^{\frac{1}{2}} \nonumber\\
    \leq &  o(1)+O\left(\int_{\mathbb R^2\setminus B_{R}(0)}  e^{\widetilde{V}_{1}} dz+o(1)\right).\nonumber
\end{align}
Letting $p\to\infty$ first and then $R\to \infty$, we obtain
\begin{align}\label{fc-4-29-0}
\lim_{p\to\infty}\int_{B_{\frac{\delta}{\varepsilon_{p,i}^*}}(0)} \left(1+\frac{v_{1,p,i}^*}{p}\right)^{\frac{p-1}{2}}\left(1+\frac{v_{2,p,i}^*}{p}\right)^{\frac{p+1}{2}} dz=0.
\end{align}
Consequently,
\begin{align}\label{fc-4-30}
    \sigma_{1,i}=&\lim_{p\to\infty}\frac{p}{M_p(y_{p,i})} \int_{B_{\delta}(y_{p,i})} \mu_1 u_{1,p}^p +\beta u_{1,p}^{\frac{p-1}{2}}u_{2,p}^{\frac{p+1}{2}} dx\nonumber\\
    =& \lim_{p\to\infty}\int_{B_{\frac{\delta}{\varepsilon_{p,i}^*}}(0)} \mu_1\left(1+\frac{v_{1,p,i}^*}{p}\right)^p+\beta\left(1+\frac{v_{1,p,i}^*}{p}\right)^{\frac{p-1}{2}}\left(1+\frac{v_{2,p,i}^*}{p}\right)^{\frac{p+1}{2}} dz\\
    =&\int_{\mathbb{R}^2}  \mu_1 e^{\widetilde{V}_{1}} dz
    =8\pi.\nonumber
\end{align}
This completes the proof.
\end{proof}

\begin{Lemma}\label{lemma-4-10=} We have
    \begin{align}
\lim\limits_{p\rightarrow+\infty} M_p(y_{p,i})=\lim\limits_{p\rightarrow+\infty} u_{1,p}(y_{p,i})=\lim\limits_{p\rightarrow+\infty} \|u_{1,p}\|_{L^\infty(B_{r_0}(x_i))}=\sqrt{e}.
    \end{align}
    Consequently, $\gamma_{1,i}=\sigma_{1,i}\lim\limits_{p\rightarrow+\infty} M_p(y_{p,i})=8\pi\sqrt{e}$.
\end{Lemma}
\begin{proof}
    By type $\mathcal{B}$, we have $v_{2,p,i}^*(0)\rightarrow-\infty$ and so
    \begin{align*}
        p(u_{2,p}(y_{p,i})-M_p(y_{p,i}))\rightarrow-\infty,
    \end{align*}
    which implies that $$u_{1,p}(y_{p,i})=M_p(y_{p,i})=\|u_{1,p}\|_{L^\infty(B_{r_0}(x_i))}.$$ Then it follows from \eqref{eq-4-37+-} that
    \begin{align}\label{eq-4-55+-}
     1
        =&\frac{1}{p}\int_{B_{{r_0}/{3\varepsilon_{p,i}^*}}(0)}H(y_{p,i},\varepsilon_{p,i}^*z+y_{p,i})g_p(z) dz\\
        &-\frac{1}{2\pi p}\int_{B_{{r_0}/{3\varepsilon_{p,i}^*}}(0)}(\log|z|) g_p(z) dz\nonumber\\
        &-\frac{1}{2\pi p}\log{\varepsilon_{p,i}^*}\int_{B_{{r_0}/{3\varepsilon_{p,i}^*}}(0)}g_p(z) dz+O\left(\frac{1}{p}\right).\nonumber
    \end{align}
    Recall \eqref{fc-4-29}-\eqref{fc-4-29-0} that
    \begin{align*}
  \lim_{p\to\infty} \int_{B_{{r_0}/{3\varepsilon_{p,i}^*}}(0)} \left(1+\frac{v_{1,p,i}^*}{p}\right)^{\frac{p-1}{2}}\left(1+\frac{v_{2,p,i}^*}{p}\right)^{\frac{p+1}{2}} dz
   =0.
    \end{align*}
Similarly, we can obtain
    \begin{align*}
    &\int_{B_{{r_0}/{3\varepsilon_{p,i}^*}}(0)} \left|\log|z|\right|\left(1+\frac{v_{1,p,i}^*}{p}\right)^{\frac{p-1}{2}}\left(1+\frac{v_{2,p,i}^*}{p}\right)^{\frac{p+1}{2}} dz\\
        \leq & o(1)+ C \left(\int_{\mathbb{R}^2\setminus B_R(0)}|\log|z||^2 e^{\widetilde{V}_{1}} dz+o(1)\right),\quad\forall R>1,
\end{align*}
and so
$$\lim_{p\to\infty}\int_{B_{{r_0}/{3\varepsilon_{p,i}^*}}(0)} \left|\log|z|\right|\left(1+\frac{v_{1,p,i}^*}{p}\right)^{\frac{p-1}{2}}\left(1+\frac{v_{2,p,i}^*}{p}\right)^{\frac{p+1}{2}} dz=0.$$
Then a similar argument as \eqref{fc-4-30} leads to
\begin{align*}
&\lim\limits_{p\rightarrow+\infty}\int_{B_{{r_0}/{3\varepsilon_{p,i}^*}}(0)}g_p(z) dz=\int_{\mathbb{R}^2} \mu_1 e^{\widetilde{V}_{1}} dz=8\pi,\\
&\int_{B_{{r_0}/{3\varepsilon_{p,i}^*}}(0)}H(y_{p,i},\varepsilon_{p,i}^*z+y_{p,i}) g_p(z) dz=O(1)\int_{B_{{r_0}/{3\varepsilon_{p,i}^*}}(0)}g_p(z) dz=O(1),\\
&\lim\limits_{p\rightarrow+\infty}\int_{B_{{r_0}/{3\varepsilon_{p,i}^*}}(0)}(\log|z|) g_p(z) dz= \int_{\mathbb{R}^2}\mu_1(\log|z|)e^{\widetilde{V}_{1}} dz=O(1).
\end{align*}
Inserting these estimates into \eqref{eq-4-55+-}, we can prove as Lemma \ref{Lemma-4-6} that $M_p(y_{p,i})\to \sqrt{e}$. This completes the proof.
\end{proof}
\begin{corollary}\label{cor-4-11}
We have that for any $\rho\in (0, r_0/3]$,
    \begin{align*}
        &\gamma_{2,i}=\lim\limits_{p\rightarrow \infty} p\int_{B_{\rho}(x_i)}\mu_2 u_{2,p}^{p} dx,\\
        &\lim\limits_{p\rightarrow \infty} p\int_{B_{\rho}(x_i)}u_{1,p}^{\frac{p+1}{2}}u_{2,p}^{\frac{p+1}{2}} dx=0,\label{eq-4-61++}\\
        &\lim\limits_{p\rightarrow \infty} p\int_{B_{\rho}(x_i)}\mu_1 u_{1,p}^{p+1} dx=8\pi e.
    \end{align*}
\end{corollary}

\begin{proof}
    Similarly as \eqref{fc-4-29-0}-\eqref{fc-4-30}, we have
    \begin{align*}
        &\lim_{p\to\infty}p\int_{B_{\rho}(x_i)}u_{k,p}^{\frac{p-1}{2}}u_{3-k,p}^{\frac{p+1}{2}} dx\\
        =& \lim_{p\to\infty} M_p(y_{p,i})\int_{B_{{\rho}/{\varepsilon_{p,i}^*}}(0)} \left(1+\frac{v_{k,p,i}^*}{p}\right)^{\frac{p-1}{2}}\left(1+\frac{v_{3-k,p,i}^*}{p}\right)^{\frac{p+1}{2}} dz
        =0,\;k=1,2,\\
        &\lim_{p\to\infty}p\int_{B_{\rho}(x_i)}\mu_1 u_{1,p}^{p+m} dx\\
        =&\lim_{p\to\infty} M_p(y_{p,i})^{1+m}\int_{B_{{\rho}/{\varepsilon_{p,i}^*}}(0)} \left(1+\frac{v_{1,p,i}^*}{p}\right)^{p+m}dz=8\pi e^{\frac{1+m}{2}},\quad m=0,1,
    \end{align*}
    which implies that
    \begin{align*}
&\gamma_{2,i}=\lim_{p\rightarrow\infty}p\int_{B_{\rho}(x_i)}\mu_2 u_{2,p}^{p} +\beta u_{1,p}^{\frac{p+1}{2}}u_{2,p}^{\frac{p-1}{2}} dx=\lim\limits_{p\rightarrow \infty} p\int_{B_{\rho}(x_i)}\mu_2 u_{2,p}^{p} dx,
    \end{align*}
and (using \eqref{eq-4.8})
    \begin{align*}
        p\int_{B_{\rho}(x_i)}u_{1,p}^{\frac{p+1}{2}}u_{2,p}^{\frac{p+1}{2}} dx=O(1)p\int_{B_{\rho}(x_i)}u_{1,p}^{\frac{p-1}{2}}u_{2,p}^{\frac{p+1}{2}} dx\to 0.
    \end{align*}
This completes the proof.    
\end{proof}

\begin{remark}
Now, comparing to the type $\mathcal{A}$ case, the new question is how to compute $\sigma_{2, i}$ or equivalently $\gamma_{2, i}$, because we have no similar estimates as \eqref{eq-4-32++-}-\eqref{eq-4-32++} for $v_{2,p, i}^*$ and so we can not use the dominated convergence theorem. We will prove below that $\sigma_{2,i}\in\{0, 8\pi\}$ and so $\gamma_{2,i}\in \{0,8\pi\sqrt{e}\}$.
\end{remark}

\begin{Lemma}
\label{cor-4-13----}
    Suppose $\sigma_{2,i}=0$, i.e. $\gamma_{2,i}=0$. Then there exists $C>0$ such that $\|p u_{2,p}\|_{L^\infty(B_{r_0}(x_i))}\leq C$ for $p$ large. Consequently, $$\lim_{p\to\infty}\|u_{2,p}\|_{L^\infty(B_{r_0}(x_i))}=0.$$
\end{Lemma}
\begin{proof} Recalling \eqref{eq-4.8} that we can take $c_0>1$ such that $\|u_{k,p}\|_{L^\infty(\Omega)}\leq c_0$ for $k=1,2$ and all $p\geq p_0$. Denote
$$
	\bar u_p:=c_0pu_{2,p}  \quad\text{and}\quad f_p:=c_0p\left(\mu_2 u_{2,p}^p +\beta u_{2,p}^{\frac{p-1}{2}}u_{1,p}^{\frac{p+1}{2}}\right)\geq 0,
$$
then
$$\begin{cases}
	-\Delta \bar u_p=f_p,\quad \bar u_p>0,\quad\text{in}~\Omega,\\
	\bar u_p=0,\quad\text{on}~\pa\Omega.
	\end{cases}$$
Note from \eqref{eq-4.9-0} that $\|f_p\|_{L^1(\Omega)}\leq C$ for all $p\geq p_0$. Since $\log t\le\frac{1}{e}t$ for any $t\in(0,+\iy)$, we obtain
	$$\log (p^{\frac{2}{p-1}}c_0u_{2,p})\le \f{1}{e}p^{\frac{2}{p-1}}c_0u_{2,p},\quad \forall x\in\Omega. $$
Then we have that for any $x\in\Omega$,
\begin{equation}\label{fc-fp1}
\begin{aligned}
	f_p&\leq Cp(c_0u_{2,p})^{\frac{p-1}{2}}=Ce^{\frac{p-1}{2}\log(p^{\frac{2}{p-1}}c_0u_{2,p})}\leq Ce^{\frac{p-1}{2e} p^{\frac{2}{p-1}}c_0u_{2,p}}\\
	&\leq Ce^{\frac1e \bar u_p},\quad\text{for $p$ large.}
	\end{aligned}
\end{equation}
Besides, we see from $\gamma_{2,i}=0$ and Corollary \ref{cor-4-11} that
\begin{equation}\label{fc-fp2}\lim_{p\to\infty}\int_{B_{r_0/3}(x_i)}f_pdx=0.\end{equation}
Thanks to \eqref{fc-fp1} and \eqref{fc-fp2}, the same proof as \cite[Lemma 2.2]{CL-JFA} implies that 
$$\|c_0pu_{2,p}\|_{L^\infty(B_{r_0/12}(x_i))}=\|\bar u_p\|_{L^\infty(B_{r_0/12}(x_i))}\leq C,\quad\text{ for $p$ large}.$$
Together with \eqref{31-1}, we finally obtain $\|pu_{2,p}\|_{L^\infty(B_{r_0}(x_i))}\leq C$ for $p$ large.
\end{proof}

\begin{Lemma}\label{lemma-4-12}
    Suppose $\sigma_{2,i}\neq 0$. Then
    $\sigma_{2,i}=8\pi$, $\gamma_{2,i}=8\pi \sqrt{e}$ and 
    \begin{equation}\label{fc-4-00}\lim\limits_{p\rightarrow+\infty}\Vert u_{2,p}\Vert_{L^\infty(B_{r_0}(x_i))}=\sqrt{e},\end{equation}
    \begin{align}\label{eq-4-99}
        \lim\limits_{p\rightarrow \infty} p\int_{B_{r_0}(x_i)}\mu_2 u_{2,p}^{p+1} dx=8\pi e.
    \end{align}
\end{Lemma}

\begin{proof}
    By Corollary \ref{cor-4-2} and Lemma \ref{Lemma-4-8}, we have $0<\sigma_{2,i}\leq 8\pi$. 
        Also recall from Lemma \ref{lemma-4-10=} that  $$0<\gamma_{2,i}=\sigma_{2,i}\lim_{p\rightarrow+\infty}M_p(y_{p,i})=\sigma_{2,i}\sqrt{e}\leq 8\pi\sqrt{e}.$$ From Lemma \ref{Lemma-4-1++}, Corollary \ref{cor-4-11} and $\gamma_{1,i}=8\pi\sqrt{e}$, we get that for any $\rho\in (0, r_0/3]$,
    \begin{align}\label{eq-4-63++}
        \left(\lim\limits_{p\rightarrow \infty} p\int_{B_{\rho}(x_i)}\mu_2 u_{2,p}^{p} dx\right)^2=\gamma_{2,i}^2=8\pi \lim\limits_{p\rightarrow \infty} p\int_{B_{\rho}(x_i)}\mu_2 u_{2,p}^{p+1} dx. 
    \end{align}
Now we claim that 
\begin{align}\label{fc-4-37}
    \lim_{p\rightarrow \infty} p\int_{B_{\rho}(x_i)}\mu_2 u_{2,p}^{p+1} dx\geq 8\pi e.
\end{align}

   Take $\chi_i \in C_0^{\infty}(B_{2\rho}(x_i))$ such that $0\leq\chi_i\leq1$ and
\be
\chi_i(x)=
\begin{cases}
  1, & \mbox{if } |x-x_i|\leq \rho \\
  0, & \mbox{if } |x-x_i|\in[\frac{4\rho}{3},2\rho).
\end{cases}\notag
\ee
Define $$\widetilde{u}_{2,p}:=u_{2,p} \chi_i \in H_0^1(B_{2\rho}(x_i)).$$ 
Then by \eqref{31-1}-\eqref{31} and Corollary \ref{cor-4-11}, we obtain
\begin{align}
p\int_{B_{2\rho}(x_i)} \widetilde{u}_{2,p}^{p+1} dx&=p\int_{B_{\rho}(x_i)} {u}_{2,p}^{p+1} dx+o(1),\\
    p\int_{B_{2\rho}(x_i)}\left|\nabla \widetilde{u}_{2,p}\right|^2 dx &=  p\int_{B_{\rho}(x_i)}\left|\nabla u_{2,p}\right|^2 dx+o(1),\label{4.8}\\
    p\int_{B_{\rho}(x_i)}\left|\nabla u_{2,p}\right|^2 dx= p\int_{B_{\rho}(x_i)} &\mu_2 u_{2,p}^{p+1}+\beta u_{2,p}^{\frac{p+1}{2}} u_{1,p}^{\frac{p+1}{2}}dx
    -p \int_{\partial B_{\rho}(x_i)}u_{2,p}\frac{\partial u_{2,p}}{\partial \Vec{n}}dS_x\notag\\
    =
    p\int_{B_{\rho}(x_i)}&\mu_2 u_{2,p}^{p+1} dx +o(1)\label{4.9}.
\end{align}
Note from \eqref{eq-4-63++} that for $p$ large,
\begin{align}
    p\int_{B_{\rho}(x_i)}  u_{2,p}^{p+1} dx &\geq  \frac{\gamma_{2,i}^2}{10\pi\mu_2}>0.\nonumber
\end{align}
Then, by applying Lemma \ref{lm2.4} with $D=B_{2\rho}(x_i)$ to $\widetilde{u}_{2,p}$, we have that for $p$ large,
\begin{align}
     p\int_{B_{2\rho}(x_i)}\left|\nabla \widetilde{u}_{2,p}\right|^2 dx
     &\geq \frac{p^{\frac{p-1}{p+1}}}{(p+1) S_p^2}
    \left[  p\int_{B_{2\rho}(x_i)} \widetilde{u}_{2,p}^{p+1} dx \right]^{\frac{2}{p+1}}\notag\\
    &= (8\pi e+o(1))\left[ p\int_{B_{\rho}(x_i)}u_{2,p}^{p+1} dx +o(1)\right]^{\frac{2}{p+1}}\geq 8\pi e+o(1) .\label{4.10}
\end{align}
This, together with \eqref{4.8} and \eqref{4.9}, proves \eqref{fc-4-37}.

Clearly \eqref{eq-4-63++}-\eqref{fc-4-37} imply $\gamma_{2,i}\geq 8\pi \sqrt{e}$ and so $\gamma_{2,i}=8\pi \sqrt{e}$, which gives $\sigma_{2,i}=8\pi$. Furthermore, \eqref{fc-4-37} becomes
\begin{align}\label{fc-4-001}8\pi e&=
    \lim_{p\rightarrow \infty} p\int_{B_{\rho}(x_i)}\mu_2 u_{2,p}^{p+1} dx\leq \lim_{p\rightarrow \infty}\Vert u_{2,p}\Vert_{L^\infty(B_\rho(x_i))}  p\int_{B_{\rho}(x_i)}\mu_2 u_{2,p}^{p} dx\\&=\gamma_{2,i}\lim_{p\rightarrow \infty}\Vert u_{2,p}\Vert_{L^\infty(B_\rho(x_i))}=8\pi\sqrt{e}\lim_{p\rightarrow \infty}\Vert u_{2,p}\Vert_{L^\infty(B_\rho(x_i))},\nonumber
\end{align}
so
    \begin{align*}
        \lim_{p\rightarrow \infty}\Vert u_{2,p}\Vert_{L^\infty(B_\rho(x_i))}\geq\sqrt{e}.
    \end{align*}
    From here and $\Vert u_{2,p}\Vert_{L^\infty(B_{r_0}(x_i))}\leq M_p(y_{p,i})\to \sqrt{e}$, we obtain \eqref{fc-4-00}. 
    Finally, \eqref{eq-4-99} follows from the first equality of \eqref{fc-4-001} and \eqref{31-1}.
    This completes the proof.
\end{proof}

Finally, analogous results as Lemma \ref{Lemma-4-8}-Lemma \ref{lemma-4-12} hold for the type $\mathcal{C}$ case and are omitted here.

\begin{proof}[Proof of Theorem \ref{thm-localmass}]
Theorem \ref{thm-localmass} follows from Lemmas \ref{lemma-4-3++++++}, \ref{Lemma-4-8} and \ref{lemma-4-12}.
\end{proof}

\section{Proof of Theorem \ref{Th-1LL}}
\label{section-5}
 In this section, we give the energy estimates of system \eqref{eq-1.1} and complete the proof of Theorem \ref{Th-1LL}.
 Define
 $$\mathcal{S}_k:=\{x_i\in\mathcal{S}\; : \; \sigma_{k,i}\neq 0\}=\{x_i\in\mathcal{S}\; : \; \gamma_{k,i}\neq 0\}, \quad k=1,2.$$
 
 \begin{Proposition}\label{cor4-16} We have that for $k=1,2$,
    \begin{align}
        p u_{k,p}\rightarrow 8\pi \sqrt{e} \sum_{x_i\in\mathcal{S}_k} G(x_i,\cdot)\quad\text{ in }\,\,C^{2}_{loc}(\overline{\Omega}\setminus\mathcal{S}_k).\label{eq-4-100++}
    \end{align}
            In particular, $\mathcal{S}_k\neq\emptyset$ for $k=1,2$.
\end{Proposition}
\begin{proof}
In Section \ref{section-4}, we have proved that $\gamma_{k,i}=8\pi \sqrt{e}$ for $x_i\in \mathcal{S}_k$ and $\gamma_{k,i}=0$ for $x_i\in\mathcal{S}\setminus\mathcal{S}_k$. Then Lemma \ref{lemma-3-7} imlies
$$ p u_{k,p}\rightarrow 8\pi \sqrt{e} \sum_{x_i\in\mathcal{S}_k} G(x_i,\cdot)\quad\text{ in }\,\,C^{2}_{loc}(\overline{\Omega}\setminus\mathcal{S}).$$
On the other hand, Lemma \ref{cor-4-13----} shows that 
$$p\Vert u_{k,p}\Vert_{L^{\infty}(\cup_{x_i\in \mathcal{S}\setminus\mathcal{S}_k}B_{r_0}(x_i))}=O(1),$$
then the same proof as Lemma \ref{lemma-3-7} actually implies that the convergence holds in $C^{2}_{loc}(\overline{\Omega}\setminus\mathcal{S}_k)$, namely \eqref{eq-4-100++} holds.

If $\mathcal S_k=\emptyset$ for some $k$, then Lemma \ref{cor-4-13----} and \eqref{31-1} together imply $p\|u_{k,p}\|_{L^\infty(\Omega)}\leq C$, a contradiction with  Proposition \ref{prop-4.1}. This proves $\mathcal{S}_k\neq\emptyset$ for $k=1,2$.
\end{proof}

Next, we give the location of the concentration points $x_i\in\s$.
\begin{Proposition} \label{prp-4-177}
For any $i=1,\cdots,N$, we have 
    \begin{align}
        \sum_{k=1}^2 m_{k,i}\left( m_{k,i} \nabla_x  R(x_i) -2\sum_{l\neq i } m_{k,l} \nabla_x G(x_i,x_l)\right)=0,
    \end{align}
    where
    \begin{align}\label{fc-mki}
        m_{k,i}:=
        \begin{cases}
            0 \quad x_i\not\in\s_k,\\
            1 \quad x_i\in\s_k.
        \end{cases}
    \end{align}
\end{Proposition}
\begin{proof}
 The proof is standard and let us take $i=1$ for example. Fix any $r\in (0, r_0/3)$. By applying the Pohozaev identity \eqref{eq-2--8} with $\Omega'=B_{r}(x_1)$, we obtain
    \begin{align}\label{fc-ppp}
        &\sum_{k=1}^2\int_{\partial B_{r}(x_1)}\frac{1}{2}|p\nabla u_{k,p}|^2  n_i -p \partial_i u_{k,p} \langle p\nabla u_{k,p}, \Vec{n}(x)\rangle dS_x\\
        =&\frac{p^2}{p+1}\int_{\partial B_{r}(x_1)}\left(\mu_1 u_{1,p}^{p+1}+\mu_2 u_{2,p}^{p+1}+2\beta u_{1,p}^{\frac{p+1}{2}}u_{2,p}^{\frac{p+1}{2}}\right) n_i dS_x,\;i=1,2.\nonumber
    \end{align}
    Again by \eqref{31-1}, we have
    \begin{align}\label{eq-4-114}
        \frac{p^2}{p+1}\int_{\partial B_{r}(x_1)}\left(\mu_1 u_{1,p}^{p+1}+\mu_2 u_{2,p}^{p+1}+2\beta u_{1,p}^{\frac{p+1}{2}}u_{2,p}^{\frac{p+1}{2}}\right) n_i dS_x=o(1).
    \end{align}
    By Proposition \ref{cor4-16}, we get that for $x\in \partial B_{r}(x_1)$,
    \begin{align}
        \frac{1}{8\pi \sqrt{e}}p \nabla u_{k,p} (x)=- \frac{m_{k,1}}{2\pi r^2}(x-x_1) -m_{k,1} \nabla_x H(x,x_1) +\sum_{l=2 }^N m_{k,l} \nabla_x G(x,x_l)+o(1),\nonumber
    \end{align}
    which implies that
    \begin{align*}
        &\frac{1}{2}|p\nabla u_{k,p}|^2 n_i -p \partial_i u_{k,p} \langle p\nabla u_{k,p}, \Vec{n}(x)\rangle\\
        =&-\frac{8 em_{k,1}^2}{r^3}(x-x_1)_i-\frac{32\pi em_{k,1}}{r}\left(m_{k,1} \nabla_x H(x,x_1) -\sum_{l=2 }^N m_{k,l} \nabla_x G(x,x_l)\right)_i+O(1).\nonumber
    \end{align*}
   Thus
    \begin{align}
        &\sum_{k=1}^2\int_{\partial B_{r}(x_1)}\frac{1}{2}|p\nabla u_{k,p}|^2  n_i -p \partial_i u_{k,p} \langle p\nabla u_{k,p}, \Vec{n}(x)\rangle dS_x
        \nonumber\\
        =& -64\pi^2 e \sum_{k=1}^2m_{k,1}\left(m_{k,1} \nabla_x H(x_1,x_1) -\sum_{l=2 }^N m_{k,l} \nabla_x G(x_1,x_l)\right)_i+O(r).\nonumber
    \end{align}
    From here and \eqref{eq-4-114}, and letting first $p\to\infty$ and then $r\to0$ in \eqref{fc-ppp}, we obtain
    \begin{align*}
        \sum_{k=1}^2 m_{k,1}\left( m_{k,1} \nabla_x  H(x_1,x_1) -\sum_{l=2 }^N m_{k,l} \nabla_x G(x_1,x_l)\right)=0.
    \end{align*} 
    The proof is complete by using $R(x)=H(x,x)$.
\end{proof}

\begin{Proposition}\label{prop-5-1}
It holds that
    \begin{align}\label{eq-5-311111}
        \lim_{p\rightarrow\infty}p\int_{\Omega}|\nabla u_{1,p}|^2+|\nabla u_{2,p}|^2 dx =8\pi e\sum_{i=1}^N (m_{1,i}+m_{2,i}),
     \end{align}
     where $m_{k,i}$ is defined in \eqref{fc-mki}.
\end{Proposition}
\begin{proof}
    By system \eqref{eq-1.1} and \eqref{31-1}, we have that for any fixed $\rho\in (0, r_0)$,
    \begin{align*}
        &p\int_{\Omega}|\nabla u_{1,p}|^2+|\nabla u_{2,p}|^2 dx\\
        =&p\int_{\om} \mu_1 u_{1,p}^{p+1}+\mu_2 u_{2,p}^{p+1}+2\beta u_{1,p}^{\frac{p+1}{2}}u_{2,p}^{\frac{p+1}{2}} dx\\
        =&\sum_{i=1}^N p\int_{B_{\rho}(x_i)}\mu_1 u_{1,p}^{p+1}+\mu_2 u_{2,p}^{p+1}+2\beta u_{1,p}^{\frac{p+1}{2}}u_{2,p}^{\frac{p+1}{2}} dx+o(1).
    \end{align*}
   Since 
  we have proved in Section 4 that  $$\mathcal{S}_k=\{x_i\in\mathcal{S}\; : \; \gamma_{k,i}\neq 0\}=\{x_i\in\mathcal{S}\; : \; \gamma_{k,i}=8\pi\sqrt{e}\}, \quad k=1,2$$  
  and $$ m_{k,i}=
        \begin{cases}
            0 \quad x_i\not\in\s_k,\\
            1 \quad x_i\in\s_k,
        \end{cases}$$    
        we have $(\gamma_{1,i},\gamma_{2,i})=(8\pi\sqrt{e}m_{1,i}, 8\pi\sqrt{e}m_{2,i})$ for any $i$. From here and \eqref{eq-4-1^*}, we obtain the desired assertion \eqref{eq-5-311111}.
\end{proof}

Now we are ready to prove Theorem \ref{Th-1LL}.

\begin{proof}[Proof of Theorem \ref{Th-1LL}]
Clearly Theorem \ref{Th-1LL} follows from Propositions \ref{cor4-16}-\ref{prop-5-1}.
\end{proof}

\subsection*{Acknowledgements}  Z.Chen is supported by National Key R\&D Program of China (Grant 2023YFA1010002) and NSFC (No. 12222109).

\end{document}